%% file: LS_CR_20220415.tex
\RequirePackage{ifpdf} 
\ifpdf
\documentclass[10pt,pdftex]{amsart}
\else
\documentclass[10pt,dvips]{amsart}
\fi

\ifpdf
\fi
\usepackage[bookmarks, 
colorlinks=false, backref,plainpages=false]{hyperref}
\usepackage[english]{babel}

\usepackage{hyperref}

\usepackage{graphics}
\usepackage{epsfig}
\usepackage{amsmath,amssymb,amsthm}

\renewcommand{\theequation}{\thesection.\arabic{equation}}

\newtheorem{thm}{Theorem}[section]

\newtheorem{lem}[thm]{Lemma}

\newtheorem{rem}[thm]{Remark}

\newtheorem{assumption}[thm]{Assumption}

\usepackage{enumerate}

\begin{document}
\input defs.tex
\def\grad{{\nabla}}

\def\calS{{\cal S}}
\def\calT{{\cal T}}
\def\cA{{\mathcal A}}
\def\cB{{\cal B}}
\def\cD{{\mathcal{D}}}

\def\cH{{\cal H}}
\def\ba{{\mathbf{a}}}

\def\bSigma{{\mathbf{\Sigma}}}

\def\beps{{\mathbf{\epsilon}}}
\def\bbbeta{{\mathbf{\eta}}}

\def\brho{\bf{\rho}}
\def\cM{{\mathcal{M}}}
\def\cN{{\mathcal{N}}}
\def\cT{{\mathcal{T}}}
\def\cE{{\mathcal{E}}}
\def\cP{{\mathcal{P}}}
\def\cF{{\mathcal{F}}}

\def\cB{{\mathcal{B}}}
\def\cG{{\mathcal{G}}}

\def\cL{{\mathcal{L}}}
\def\cJ{{\mathcal{J}}}
\def\cV{{\mathcal{V}}}
\def\cW{{\mathcal{W}}}

\def\tbeta{{\mathtt{\beta}}}
\def\tv{{\mathtt{v}}}
\def\tq{{\mathtt{q}}}
\def\tcurl{{\mathtt{Curl\,}}}
\newcommand{\lJump}{[\![}
\newcommand{\rJump}{]\!]}
\newcommand{\jump}[1]{[\![ #1]\!]}

\newcommand{\sd}{\bsigma^{\Delta}}
\newcommand{\rd}{\brho^{\Delta}}

\newcommand{\eps}{\epsilon}
\newcommand{\R}{{\mathbb R}}
\newcommand{\M}{{\mathbb M}}
\newcommand{\T}{{\mathbb T}}
\newcommand{\Y}{{\mathbb Y}}
\newcommand{\Z}{{\mathbb Z}}
\newcommand{\Q}{{\mathbb Q}}

\newcommand{\BNB}{{Banach-Nec\v{a}s-Babu\v{s}ka }}

\title [Nonconforming LSFEMs for General Elliptic Equations]{Least-Squares  Methods with Nonconforming Finite Elements for General Second-Order Elliptic Equations}
\author[Y. Liang and S. Zhang]{Yuxiang Liang and Shun Zhang}
\address{Department of Mathematics, City University of Hong Kong, Kowloon Tong, Hong Kong SAR, China}
\email{yuxiliang7-c@my.cityu.edu.hk, shun.zhang@cityu.edu.hk}
\thanks{This work was supported in part by
Research Grants Council of the Hong Kong SAR, China under the GRF Grant Project No. CityU 11302519,  CityU 11300517, and CityU 11305319}
\date{\today}


\maketitle
\begin{abstract}
In this paper, we study least-squares finite element methods (LSFEM) for general second-order elliptic equations with nonconforming finite element approximations. The equation may be indefinite. For the two-field potential-flux div LSFEM with Crouzeix-Raviart (CR) element approximation, we present three proofs of the discrete solvability under the condition that mesh size is small enough. One of the proof is based on the coerciveness of the original bilinear form. The other two are based on the minimal assumption of the uniqueness of the solution of the second-order elliptic equation. A counterexample shows that div least-squares functional does not have norm equivalence in the sum space of $H^1$ and CR finite element spaces. Thus it cannot be used as an a posteriori error estimator. Several versions of reliable and efficient error estimators are proposed for the method. We also propose a three-filed potential-flux-intensity div-curl least-squares method with general nonconforming finite element approximations. The norm equivalence in the abstract nonconforming piecewise $H^1$-space is established for the three-filed formulation on the minimal assumption of the uniqueness of the solution of the second-order elliptic equation. The three-filed div-curl nonconforming formulation thus has no restriction on the mesh size, and the least-squares functional can be used as the built-in a posteriori error estimator. Under some restrictive conditions, we also discuss a potential-flux div-curl least-squares method.
\end{abstract}

\section{Introduction}\label{intro}
\setcounter{equation}{0}
The least-squares variational principle and the corresponding least-squares finite element methods based on a first-order system reformulation have been widely used in numerical solutions of partial differential equations; see for example \cite{CLMM:94,CMM:97,Jiang:98,BG:09,CS:04,CLW:04,CFZ:15,LZ:18,LZ:19,QZ:20}. Compared to the standard variational formulation and the related finite element methods, the first-order system least-squares finite element methods have several known advantages, such as the discrete problem is stable without the inf-sup condition of the discrete spaces and mesh size restriction, and the least-squares functional itself is a good built-in a posteriori error estimator.

Since the introduction in the classic 1973 paper of Crouzeix and Raviart \cite{CR:73}, the nonconforming finite elements, including the Crouzeix-Raviart (CR) element \cite{Brenner:15}, and various elements introduced in \cite{FS:83,Han:84,RT:92,DSSX:99,CDY:99} are very useful for numerical computation of many physical problems. In \cite{DL:03},  least-squares methods with various nonconforming finite element approximations are introduced for the diffusion problem without the lower order terms. In \cite{DL:03}, contradicting least-squares methods with conforming approximations, the coerciveness of the discrete problem needs to be proved independently. The discrete problem is coercive for the CR element for the equation without the lower-order terms. For some other elements, the assumption that the mesh size is sufficiently small is needed to ensure the coerciveness in \cite{DL:03}. The a posteriori error estimator is not discussed in \cite{DL:03}. 

\subsection{Stability results for various finite element approximations to second-order linear elliptic PDEs}
In this paper, we plan to extend the least-squares methods with nonconforming finite element approximations to general second-order elliptic equations. We first review the existence and uniqueness result of the general second-order elliptic equation and its various finite element approximations.

For the general linear second-order elliptic equation \eqref{pde1}, the solution's existence and uniqueness can be discussed in two cases. 
The first case is a simpler coercive case. Assuming that the coefficients satisfy some assumption, the standard bilinear form is coercive in the $H^1$-norm; see also Section 6.1.2 of \cite{QV:94}. The existence, uniqueness, and stability of the solution can be obtained from the Lax-Milgram Lemma. For its conforming finite element approximation, the method is also coercive.  

The second case is the general case, where the bilinear form associated with the PDE can be genuinely indefinite. In the general case, we only need to assume a very mild assumption of the domain and the coefficients; see Assumption \ref{ass_dc}. Furthermore, due to the compactness of the operator, the uniqueness, existence, and well-posedness are equivalent; see discussion in Section 2.2. We thus only assume a {\it minimal assumption: the homogeneous equation has a unique zero solution} to ensure the well-posedness of the problem. The coercive case is a special case of the general case. 

The analysis of finite element approximations to the possible indefinite general second-order elliptic equations is not simple and straightforward. In \cite{Sch:74,SW:96}, conforming finite element approximations of general elliptic equations are discussed. Recently, in \cite{CDNP:16,CNP:22}, nonconforming and mixed finite element approximations are discussed. The results require some regularity assumption and that the mesh size of the discretization is small enough. On the other hand, once the coerciveness of the least-squares method is proved, see \cite{CLMM:94,BLP:97,Cai:04,CFZ:15,Ku:07,Zhang:22}, the LSFEM with $H^1$-$H(\divvr)$ conforming finite element approximations is automatic coercive, without the restriction on the mesh size.

\subsection{Contributions of this paper}
For the CR-LSFEM for general second-order elliptic equations, we first present a two-field potential-flux div least-squares method, which is a direct generalization of the nonconforming method suggested in \cite{DL:03} to general second-order elliptic equations and a generalization of the conforming least-squares method in \cite{CLMM:94} to the CR approximation.

In the paper, we present a negative result on the norm equivalence of the nonconforming div least-squares functional in the sum space of $H^1$ and CR finite element spaces; see Lemma \ref{counter_ex}. Thus, the two-field potential-flux div CR-LSFEM does not have the two most important properties of a standard LSFEM: automatic discrete stability without the assumption of the mesh size and a built-in least-squares functional a posteriori error estimator. We need to prove the discrete solvability and develop a posteriori error estimator for the method.

To ensure the solvability of the discrete problem from the two-field CR-LSFEM, we discuss three proofs. The first proof is for the coercive case only. The second and the third proofs are for the general cases under the minimal uniqueness assumption.  Our second proof is based on Schatz's argument. We present an $L^2$-error estimate of the potential-flux div CR-LSFEM based on Cai-Ku's paper \cite{CK:06}. The proof also corrects a small error in the original \cite{CK:06} when handling mixed boundary conditions. In the third proof, we present a very short proof by combining the new proof presented in \cite{Zhang:22} and the discrete stability result proved in \cite{CDNP:16} for CR approximation of the general elliptic equation.  Different from the conforming LSFEM in \cite{CLMM:94}, in the second and third proofs, the div CR-LSFEM requires the same regularity and sufficient small mesh size assumptions as the non-least-squares (conforming, nonconforming, and mixed) methods \cite{Sch:74,SW:96,CDNP:16,CNP:22}.
In the first proof (coercive case), the restriction of the mesh size is local, and no regularity assumption is needed. The mesh size restriction in the second and third proofs is global since the regularity assumption is global.  
 
 For a posteriori error estimator of two-field potential-flux div CR-LSFEM, we suggest several a posteriori error estimators for the potential-flux div CR-LSFEM by adding different terms measuring the nonconforming error. Reliability and efficiency results are proved. 

To overcome the shortcoming of the two-field div CR-LSFEM due to the lack of norm equivalence, we suggest three-field formulations, potential-flux-density div-curl least-squares methods. A curl term of the intensity (gradient of the solution) is added to the least-squares formulation. In this new formulation, we prove norm equivalence for the abstract nonconforming piecewise $H^1$-space. Since the nonconforming space is mesh-dependent, we use Helmholtz decomposition to avoid the coerciveness constant depending on the mesh. The coerciveness of potential-flux-density div-curl least-squares methods in the abstract nonconforming space setting is then proved in the same minimal uniqueness assumption as the standard least-squares formulation without any requirement on the mesh size. Since the norm equivalence in the abstract nonconforming piecewise $H^1$-space is true, we automatically have the standard built-in least-squares a posteriori error estimator. The only other ingredient needed for the proof is a discrete Poincar\'e-Friedrichs inequality in the abstract nonconforming piecewise $H^1$-space. Thus the proof is not only true for the CR element but also true for nonconforming elements  introduced in \cite{FS:83,Han:84,RT:92,DSSX:99,CDY:99}. In a sense, we recover the good properties of the LSFEM (non-restriction on mesh size and built-in a posteriori estimator) over the conforming, nonconforming, and mixed methods with the three-filed nonconforming LSFEM.

The failure and success of the norm equivalence with and without the curl term can be connected to the theory of a posteriori error estimates of nonconforming finite elements. For the nonconforming finite element approximation, besides the standard residual and the fact the numerical flux from the potential $u_h$ is not in the $H(\divvr)$ space (the so-called conforming error), we also need to measure one extra error:  the numerical potential $u_h$ is not in the $H^1$ space or its consequence that the numerical intensity (gradient) is not in the $H(\curll)$ space (the so-called nonconforming error). This explains that we need the intensity (gradient) as an independent unknown and add its curl in the potential-flux-density div-curl least-squares method. Previous discussions of  a posteriori error estimates of the nonconforming finite element approximation can be found in \cite{DDPV:96,CBJ:02,Ain:05,CZ:10a,CHZ:17cr,CHZ:17,CHZ:20,CHZ:21}, where the different contributions of the error are discussed.

We also discuss the application and restriction of the original potential-flux div-curl least-squares method \cite{CMM:97}. When the domain is nice and the coefficient is sufficiently smooth, the original formulation introduced in \cite{CMM:97} can be used in the nonconforming case. The norm equivalence can also be established similarly. However, the two-field div-curl formulation can cause serious problems when the conditions on the domain and coefficients are not satisfied. 

\subsection{Structure of the paper}
The paper is organized as follows. In Section 2, we present preliminaries about the abstract and discrete spaces. Properties of the CR space, the Helmholtz decomposition, and the discrete Poincar\'e-Friedrichs inequality are discussed. The solution theory of a general second-order elliptic equation is also discussed. Sections 3 to 7 are about the two-field potential-flux div CR-LSFEM. In Section 3, we introduce the formulation. A discrete coerciveness is proved with the assumption of the coefficients ensuring the coerciveness of the bilinear form of the original variational problem in Section 4. We present two proofs of the discrete solvability with minimal uniqueness assumption and a regularity assumption in Sections 5 and 6.  In Section 7, we present a counterexample to show that div least-squares functional does not have norm equivalence in the sum space of $H^1$ and CR finite element spaces and we propose several versions of reliable and efficient error estimators. We propose the potential-flux-intensity div-curl least-squares method with general nonconforming finite element approximations in Section 8. The norm equivalence in the abstract nonconforming piecewise $H^1$-space is established. Under some restrictive conditions, we also discuss the potential-flux div-curl least-squares method in Section 9. Several concluding remarks are made in Section 10.

\section{Preliminaries}
\setcounter{equation}{0}
\subsection{Notations and the function spaces}
\setcounter{equation}{0}
Let $\O$ be a bounded, open, connected subset of $\mathbb{R}^d (d = 2 \mbox{ or } 3)$ with a Lipschitz continuous boundary $\p\O$. We partition the boundary of the domain $\O$ into two open subsets $\G_D$ and $\G_N$, such that $\p\O = \overline{\G_D} \cup \overline{\G_N}$ and $\G_D\cap \G_N =\emptyset$. For simplicity, we assume that $\G_D$ is not empty (i.e., $\mbox{meas}(\G_D) \neq 0$ ) and is connected.

We use the standard notations and definitions for the Sobolev spaces $H^s(\O)^d$ for $s\ge 0$. The standard associated inner product is denoted by $(\cdot , \,
\cdot)_{s,\O}$, and their respective norms are denoted by $\|\cdot \|_{s,\O}$ and
$\|\cdot\|_{s,\partial\O}$. The notation $|\cdot|_{s,\O}$ is used for semi-norms.  (We suppress the superscript $d$ because the dependence on dimension will be clear by context. We also omit the subscript $\O$ from the inner product and norm designation when there is no risk of confusion.) For $s=0$,
$H^s(\O)^d$ coincides with $L^2(\O)^d$.  The symbols $\gradt$ and $\nabla$ stand for the divergence and gradient operators, respectively. Set $H^1_D(\Omega):=\{v\in H^1(\Omega)\, :\, v=0\,\,\mbox{on }\Gamma_D\}$, $H^1_N(\Omega):=\{v\in H^1(\Omega)\, :\, v=0\,\,\mbox{on }\Gamma_N\}$, and $H^1_0(\Omega):=\{v\in H^1(\Omega)\, :\, v=0\,\,\mbox{on }\p\O\}$.

In two dimensions, for a vector-valued function $\btau= (\tau_1,\,\tau_2)^t$, define the curl operator by  $\curlt \btau := \dfrac{\p\tau_2}{\p x_1} - \dfrac{\p\tau_1}{\p x_2}$. For a scalar-valued function $v$, define the operator $\gperp$ by $\gperp v = (-\dfrac{\p v}{\p x_2},\,\dfrac{\p v}{\p x_1})$. 
In three dimensions, $\curlt \btau$ is defined standardly for a vector valued function $\btau$. To unify the notation in both dimensions, we define the vector curl as following. Let $k := 1$ if $d = 2$ and $k := 3$ if $d = 3$. The Curl of a function $\mathtt{v}\in \R^k$ is defined by
\beq
\tcurl \mathtt{v}:= \gperp \tv \mbox{ if } d=2 \quad \mbox{and}\quad
\tcurl \mathtt{v}:= \curlt \tv \mbox{ if } d=3.
\eeq
We use a special font for $\tv$  that it is a scalar function when $d=2$ and it is a vector function when $d=3$.
Given a unit normal $\bn$ we define the tangential component of a vector $\bv\in \R^d$ with respect to $\bn$ by
\beq
\gamma_{t}(\bv) := \left\{
\begin{array}{ccc}
\bv\cdot \bt & \mbox{if }d=2, & \mbox{ where }\bt = (-n_2,n_1) \mbox{ if }\bn = (n_1.n_2),\\[1mm]
\bv\times \bn & \mbox{if } d=3.&
\end{array}
\right.
\eeq
We use the standard 
$
H(\divvr;\O)$ and $ H(\curll;\O)$ spaces,
equipped with the norms
\[
\|\btau\|_{H(\divvr;\,\O)}=\left(\|\btau\|^2_{0,\O}+\|\gradt\btau\|^2_{0,\O}
 \right)^\frac12
 \quad\mbox{and}\quad
 \|\btau\|_{H(\curll;\,\O)}=\left(\|\btau\|^2_{0,\O}+\|\curlt\btau\|^2_{0,\O}
 \right)^\frac12,
\]
respectively.
Denote their subspaces by
 \[
 H_N(\divvr;\O)=\{\btau\in H(\divvr;\O)\, :\,
\btau\cdot\bn|_{\Gamma_N}=0\}
\quad \mbox{and}\quad 
 H_D(\curll;\O)=\{\btau\in H(\curll;\O)\, :\,\gamma_{t}(\btau)|_{\Gamma_D}=0\},
\]
where $\bn$ is the unit vectors outward normal to the boundary $\p\O$.

Let $\cT = \{K\}$ be a triangulation of $\O$ using simplicial elements. The mesh $\cT$ is assumed to be regular.  Denote the set of all nodes of the triangulation by
$
 \cN := \cN_{int}\cup\cN_{D}\cup\cN_{N},
$
where $\cN_{int}$ is the set of all interior nodes, and $\cN_D$ and $\cN_{N}$ are the sets of all boundary nodes belonging to the respective $\overline{\Gamma}_D$ and $\Gamma_N$.
Denote the set of all faces(3D)/edges(2D) of the triangulation by
$
 \cE := \cE_{int}\cup\cE_{D}\cup\cE_{N},
$
where $\cE_{int}$ is the set of all interior element faces/edges and $\cE_D$ and $\cE_{N}$ are the sets of all boundary faces/edges belonging to the respective $\Gamma_D$ and $\Gamma_N$. For each $F \in \cE$, denote by $h_F$ the diameter of the face/edge $F$; denote by $\bn_F$ a unit vector normal to $F$. When $F \in \cE_D \cup \cE_{N}$, assume that $\bn_F$ is the unit outward normal vector. For each interior face/edge $F\in\cE_{int}$, let $K_F^{+}$ and $K_F^{-}$ be the two elements sharing the common edge $F$ such that the unit outward normal vector of $K_F^{-}$ coincides with $\bn_F$.

Define jumps and averages over faces/edges by
$$
\jump{v}_F :=  \left\{
\begin{array}{lll}
 v|^{-}_F - v|_F^{+} \quad & F\in \cE_{int},\\[2mm]
v|_F  \quad & F\in \cE_{D}\\[2mm]
0  \quad & F\in \cE_N,
\end{array}
\right.
\quad\mbox{and}\quad
\{v(x)\}_F =  \left\{
\begin{array}{lll}
 (v^-_F+v^+_F)/2  & F\in \cE_{int}, \\[2mm]
v|_F  \quad & F\in \cE_{D}\cup\cE_N
\end{array}
\right.
$$
for all $F\in\cE$. 
A simple calculation leads to the
following identity:
$
 \jump{ u v}_F = \{v\}_F\, \jump{u}_F + \{u\}_F\, \jump{v}_F$,  for all $F\in\cE_{int}.
$
Let $P_k(K)$ be the space of polynomials of degree $k$ on element $K$. Denote the Crouzeix-Raviart nonconforming piecewise linear finite element spac \cite{CR:73,Brenner:15} associated with the triangulation $\cT$ by
\[
 V^{cr}= \{v\in L^2(\O)\,:\,v|_K\in  P_1(K)\,\,\forall\,\,K\in\cT
 \mbox{ and } \int_F\jump{v}ds = 0  \;\forall\, F\in
 \cE_{int}\}
\]
and its subspace by
$
 V^{cr}_D=\{v\in V^{cr}\,:\, \int_F v ds =0 \,\,\forall\,\, F\in
 \cE_D\}.
$
Let 
\begin{eqnarray}
W_D^{1+cr}&:=& H^1_D(\O) +  V^{cr}_D =\{v = v_1+v_2: v_1\in H^1_D(\O), v_2 \in V^{cr}_D  \},\\
W_D(\cT) &: =&\{v \in L^2(\O):  v|_K \in H^1(K), \forall K \in \cT \mbox{ and } \int_F\jump{v}ds = 0  \;\forall\, F\in \cE_{int}\cup\cE_D  \}.
\end{eqnarray}
It is easy to see that 
$
W_D^{1+cr}\subset W_D(\cT).
$
Many other classic nonconforming elements \cite{FS:83,Han:84,RT:92,DSSX:99,CDY:99} belong to $W_D(\cT)$. 
 
Denote the local lowest-order Raviart-Thomas (RT) \cite{RT:77} on element $K\in\cT$ by $RT_0(K)=P_0(K)^d +\bx\,P_0(K)$. Then the standard lowest-order $\Hdiv$ conforming RT space is defined by
 \[
 RT_{0,N}=\{\btau\in H_N(\divvr;\,\Omega) \,:\,
 \btau|_K\in RT_0(K)\,\,\,\,\forall\,\,K\in\cT\}.
 \] Also, let
$$
P_0 =\{v\in L^2(\O)\,:\, v|_K \in P_0(K) \,\, \forall\,\, K\in
\cT\}.
$$
Denote the first type of local lowest order N\'ed\'elec space \cite{Ned:80} on element $K\in\cT$ by
$$
N_0(K):= \left\{ \begin{array}{lll}
P_0(K)^2+(x_2,-x_1) P_0(K) & d=2,\\
P_0(K)^3+(0,-x_3,x_2) P_0(K)+(x_3,0,-x_1) P_0(K)+(-x_2,x_1,0) P_0(K) & d=3.
\end{array}
\right.
$$
Then the lowest order  $\Hcrl$ conforming N\'ed\'elec spaces are defined by
$$
 N_{0,D}=\{\btau\in H_D(\curll\,;\,\Omega) \,:\, \btau|_K\in  N_0(K)\,\,\,\,\forall\,\,K\in\cT\}.
$$
We define the discrete gradient operator 
as $(\nabla_h v)|_K := \grad(v|_K)$, for all $K\in \cT$.

For a fixed $r>0$, denote by $I_{rt}: \Hdiv \cap [H^r(\O)]^d \mapsto RT_0$ the standard $RT$ interpolation operator.  We have the following local approximation property: for $\btau \in H^{\ell_K}(K)$, $0<\ell_K \leq 1$,
\begin{eqnarray} \label{rti}
\|\btau - I_{rt} \btau\|_{0,K}
  &\leq& C h_K^{\ell_K} |\btau|_{\ell_K,K} \quad\forall\,\, K\in \cT, \\[2mm]
  \label{divrti}
  \|\gradt(\btau - I_{rt} \btau)\|_{0,K}
  &\leq& C h_K^{\ell_K} |\gradt\btau|_{\ell_K,K} \quad\forall\,\, K\in \cT.
\end{eqnarray}
The estimates in \eqref{rti} and \eqref{divrti} are standard for $\ell_K= 1$ and can be proved by the average Taylor series developed in \cite{DuSc:80} and the standard reference element technique with Piola transformation for $0<\ell_K<1$. The interpolations and approximation properties are entirely local.

Similarly, denote by $I_{n}: \Hcrl \mapsto N_0$ the standard lowest-order N\'ed\'elec interpolation operator.  We have the following approximation property: for $\btau \in H^1(\O)$ and $\nabla \times\btau \in  H^1(\O)^d$, 
\begin{eqnarray} \label{ndi}
\|\btau - I_{n} \btau\|_{0}  &\leq& C h |\btau|_{1}  \\[2mm]
  \label{curlndi}
  \|\nabla\times (\btau - I_{n} \btau)\|_{0}
  &\leq& C h_K |\nabla \times\btau|_{0}. 
\end{eqnarray}

Denote by $\theta_F (\bx)$ the nodal basis function of $V^{cr}_D$ associated with the face
$F\in\cE$. For $v\in L^1(F)$, define $\Pi_F^0 v = (v,1)_F/|F|$, the average value of $v$ on $F$. The local Crouzeix-Raviart interpolant is defined by $ I^{cr}_K v = \sum_{F\in \cE \cap \partial K} (\Pi_F^0 v) \theta_F(\bx)$ for  $v\in W^{1,1}(K)$.
It was shown (see \cite{EG:04,CHZ:17}) that for $v\in H^{1+\ell_K}(K)$ with $0\leq \ell_K  \leq 1$
\beq \label{localcr}
 \|v-I^{cr}_K v\|_{0,K}  \leq C \, h_{K}^{1+\ell_K}\, |\nabla v|_{\ell_K,K}\quad\forall\,\, K\in \cT. 
\eeq

\subsection{General second-order elliptic equations}
Consider the general second-order elliptic equation in divergence form
\beq \label{pde1}
-\gradt(A \nabla u) + X u=f  \mbox{ in } \O, \quad
u = 0 \mbox{ on } \Gamma_D, \quad
A \nabla u \cdot\bn = 0 \mbox{ on } \Gamma_N.
\eeq
where 
$$
Xv := \bb\cdot\nabla v+cv, \quad \forall v\in H^1(\O).
$$
  
The following very mild conditions on the domain and coefficients are assumed.
\begin{assumption}\label{ass_dc}
The domain $\O$ is a bounded, open, connected subset of $\mathbb{R}^d (d = 2 \mbox{ or } 3)$ with a Lipschitz continuous boundary $\p\O$.
The diffusion coefficient matrix $A \in L^{\infty}(\O)^{d\times d}$ is a given $d\times d$ tensor-valued function;  the matrix $A$ is uniformly symmetric positive definite: there exist positive constants $0 < \Lambda_0 \leq \Lambda_1$ such that
\beq\label{A}
\Lambda_0 \by^T\by \leq \by^T A \by \leq \Lambda_1 \by^T\by
\eeq
for all $\by\in \mathbb{R}^d$ and almost all $x\in \O$. 
The coefficients $\bb \in L^{\infty}(\O)^d$ and  $c\in L^{\infty}(\O)$ are given vector- and scalar-valued bounded functions, respectively.
\end{assumption}

The variational problem of \eqref{pde1} is:  Find $u \in H^1_D(\O)$, such that 
\beq\label{vp1}
a(u,v) = (f,v) \quad\forall v\in H^1_D(\O),
\eeq
where the bilinear form $a$ is defined as 
\beq \label{a}
a(w,v) := (A\nabla w,\nabla v) +(X w,v)\quad \mbox{for } w, v\in H^1_D(\O).
\eeq
It is easy to check that the bilinear form is continuous due to Assumption \ref{ass_dc}:
\beq
a(w,v) \leq C\|\nabla v\|_0 \|\nabla w\|_0 \quad v\in  H^1_D(\O) \mbox{ and } w\in  H^1_D(\O). 
\eeq
We also consider the adjoint problem of \eqref{pde1}, a PDE in physical form,
\beq \label{pde2}
\begin{array}{rcl}
-\gradt(A \nabla z+\bb z) + c z &=& g \mbox{ in } \O, \\[1mm]
z &=& 0 \mbox{ on } \Gamma_D, \\[1mm]
(A \nabla z +\bb z)\cdot\bn &=& 0 \mbox{ on } \Gamma_N.
\end{array}
\eeq
With the help of integrations by parts, it is easy to check that the variational problem of \eqref{pde2} is: Find $z \in H^1_D(\O)$, such that 
\beq\label{vp2}
a(v,z) = (g,v) \quad\forall v\in H^1_D(\O).
\eeq
Thus  \eqref{pde2} is the adjoint problem of \eqref{pde1}.

\begin{rem}
When $\bb=0$ and $c = -k^2$ for some $k>0$, the equation is a real Helmholtz equation. The equation is indefinite. However, as long as $k^2$ square is not an eigenvalue of $(A\nabla v,\nabla w)$, it still has a unique solution.
\end{rem}
To discuss the solution theory of linear second-order elliptic equations, we consider two slightly more general problems with righthand sides in $(H^1_D(\O))'$, the dual space of $H^1_D(\O)$.

For a $\phi\in (H^1_D(\O))'$, assume that $u\in H^1_D(\O)$ is the solution of the weak problem of the equation in divergence form:
\beq\label{vp11}
\mbox{Find } u\in H^1_D(\O), \mbox{ such that }
a(u,v) = \langle \phi, v\rangle_{(H^1_D(\O))'\times H^1_D(\O)} \quad\forall v\in H^1_D(\O),
\eeq
or the adjoint weak problem of the equation in physical form, a $\psi\in (H^1_D(\O))'$
\beq\label{vp22}
\mbox{Find } z\in H^1_D(\O), \mbox{ such that }
a(v,z) = \langle \psi, v\rangle_{(H^1_D(\O))'\times H^1_D(\O)}  \quad\forall v\in H^1_D(\O).
\eeq

For a possible indefinite linear second-order elliptic equation, the solution's existence and uniqueness theory is based on the Fredholm alternative. Since we assume the ellipticity of the PDEs (conditions \eqref{A} on  $A$), the operators associated with the divergence form and adjoint physical form problems are Fredholm operators of index zero. The uniqueness, existence, and well-posedness are equivalent for the linear second-order elliptic equation, for example, see discussions in various standard PDE books \cite{BJS:64,GT:01,Evans:10}. A detailed discussion of the following theorem with two proofs can also be found in Theorem 2.3 of \cite{Zhang:22}. Some more discussion can also be found in the introductions of \cite{CDNP:16,CNP:22}. 

\begin{thm}\label{assp}
Assume Assumption \ref{ass_dc} is true. The following assumptions are equivalent:
\begin{enumerate}[(1)]

\item The homogeneous equation $a(u,v) = 0$, for all $v\in H^1_D(\O)$, has $u=0$ as its unique solution. 

\item 
The weak problem \eqref{vp11} has the following stability bound:
\beq \label{apriori1}
	\|\nabla u\|_0 \leq C\|\phi\|_{(H^1_D(\O))'} \quad \forall \phi\in (H^1_D(\O))'.
\eeq

\item The weak problem \eqref{vp11} has a unique solution $u\in H^1_D(\O)$ for any $\phi\in (H^1_D(\O))'$. 

\item The homogeneous equation associated to \eqref{vp22}, i.e., $a(v,z) = 0$, for all $v\in H^1_D(\O)$, has $z=0$ as its unique solution.

\item 
The adjoint weak problem \eqref{vp22} has the following stability bound:
\beq \label{apriori2}
	\|\nabla z\|_0 \leq C\|\psi\|_{(H^1_D(\O))'} \quad \forall \psi\in (H^1_D(\O))'.
\eeq

\item The adjoint weak problem \eqref{vp22} has a unique solution $z\in H^1_D(\O)$ for any $\psi\in (H^1_D(\O))'$. 
\end{enumerate}
\end{thm}
In this paper, except for the case in Section 4, we only assume that one of the conditions in Theorem \ref{assp} is true.

%
%

\subsection{First-order systems of general second-order elliptic equation}
Following the notations of \cite{BG:09}, we call the solution $u$ as potential. Let the flux $\bsigma= -A\nabla u$. We have the following two-field potential-flux first-order system:
\begin{equation} \label{1os}
\left\{
\begin{array}{lllll}
\gradt \bsigma + Xu    & =& f & \mbox{in } \O
 \\[1mm]
A\nabla u+ \bsigma  & =& 0 & \mbox{in } \O,
\end{array}
\right.
\end{equation}
with boundary conditions
$
u=0 \mbox{ on }\Gamma_D
\mbox{ and } \bn\cdot\bsigma=0 \mbox{ on } \Gamma_N.
$
We have $\bsigma \in H_N(\divvr\;\O)$ and $u\in H^1_D(\O)$.

Introduce the intensity $\bphi = -\nabla u$. Then we have the following three-field potential–flux–intensity $(u,\bsigma,\bphi)$ div-curl first-order (redundant) system as seen in Chapter 5 of \cite{BG:09}:
\begin{equation} \label{gen_sys}
\left\{
\begin{array}{lllll}
\gradt \bsigma +  Xu  & =& f & \mbox{in } \O
 \\[1mm]
\curlt \bphi &=& 0 & \mbox{in } \O ,\\[1mm]
A\nabla u+ \bsigma  & =& 0 & \mbox{in } \O, \\[1mm]
A\bphi - \bsigma &=& 0 & \mbox{in } \O,\\[1mm]
\nabla u+ \bphi  & =& 0 & \mbox{in } \O,
\end{array}
\right.
\end{equation}
with boundary conditions
$
u=0 \mbox{ on }\Gamma_D$,  $\gamma_{t}(\bphi)=0 \mbox{ on }\Gamma_D$, 
and $\bn\cdot\bsigma=0 \mbox{ on } \Gamma_N.
$
We have the flux $\bsigma\in H_N(\divvr;\O)$, the intensity $\bphi \in H_D(\curllr;\O)$, and the potential $u\in H^1_D(\O)$.

\subsection{Some results for Crouzeix-Raviart elements}
Let $S_{2,D}\subset H_D^1(\O)$ be the conforming $P_2$ Lagrange finite element space associated with the mesh $\cT$. Define the following enriching operator $E_h: V^{cr}_D \rightarrow S_{2,D}$ by averaging:
\beq\label{enrich}
(E_h v)(z) = \dfrac{1}{|\cT_z|}\sum_{K\in\cT_z} v|_K (z) \quad \forall z\in \cN_{int}\cup\cN_N,
\eeq
where $\cT_z$ is the set of the elements in $\cT$ that share $z$ as a common vertex and $\cT_z$ is the number of the elements in $\cT_z$.  We have the following estimate (see (2.27), (2.37), and (2.28) of \cite{Brenner:15}):
\begin{eqnarray}\label{esE1}
\sum_{K\in\cT}h_K^{-2} \|v_{cr}-E_h v_{cr}\|_{0,K}^2 &\leq& C \sum_{F\in\cE}\dfrac{1}{h_F}\|\jump{v_{cr}}\|_{0,F}^2
\quad \forall v_{cr}\in V_D^{cr}, \\ \label{esE2}
\|\nabla_h(v_{cr}-E_h v_{cr})\|_{0}^2 &\leq& C \sum_{F\in\cE}\dfrac{1}{h_F}\|\jump{v_{cr}}\|_{0,F}^2
\quad \forall v_{cr}\in V_D^{cr},\\
\|\nabla (E_h v_{cr})\|_{0,\O} &\leq & C \|\nabla_h v_{cr}\|_{0,\O}\quad \forall v_{cr}\in V_D^{cr}.
\end{eqnarray}
Although the above estimates are proved in \cite{Brenner:15} for pure Dirichlet boundary condition only, it is not hard to see they are also true for mixed boundary conditions with a non-empty $\G_D$.

Following the argument in (2.41) of \cite{Brenner:15}, for  $v_{cr}\in V_D^{cr}$ and $v\in H_D^1(\O)$,
we also have
\beq\label{ejump}
\dfrac{1}{h_F} \|\jump{v_{cr}}\|_{0,F}^2 
= \dfrac{1}{h_F}\|\jump{v_{cr}-v}-\Pi_F^0\jump{v_{cr}-v}\|_{0,F}^2 
\leq C\|\nabla_h(v-{v_{cr}})\|_{0,K_F^+\cup K_F^-}^2.
\eeq
In other words, 
\beq  \label{esE4}
\sum_{F\in\cE} \dfrac{1}{h_F} \|\jump{v_{cr}}\|_{0,F}^2 \leq C\inf_{v\in H^1_D(\O)}\|\nabla_h(v-{v_{cr}})\|_0^2 \quad \forall v_{cr}\in V_D^{cr}.
\eeq
Specifically, let $v=0$ in \eqref{esE4}, 
\beq
\sum_{F\in\cE} \dfrac{1}{h_F} \|\jump{v_{cr}}\|_{0,F}^2 \leq C\|\nabla_h{v_{cr}}\|_0^2 \quad \forall v_{cr}\in V_D^{cr}.
\eeq
Choosing $v=u$, the solution of \eqref{pde1}, in \eqref{esE4}, we have
\beq \label{jumperror}
\sum_{F\in\cE} \dfrac{1}{h_F} \|\jump{v_{cr}}\|_{0,F}^2 \leq C\|\nabla_h(u-{v_{cr}})\|_0^2 \quad \forall v_{cr}\in V_D^{cr}.
\eeq

We also have the following equivalence between the function jump and tangential jump of the discrete gradient: For $F\in\cE$ and $v_{cr}\in V^{cr}_D$, we have
\begin{eqnarray}\label{jj1}
C h_F \|\jump{\gamma_{t_F}(\nabla v_{cr})}\|_{0,F}^2 \leq 
\dfrac{1}{h_F} \|\jump{v_{cr}}\|_{0,F}^2 \leq C h_F \|\jump{\gamma_{t_F}(\nabla v_{cr})}\|_{0,F}^2.
\end{eqnarray}
The equivalence in two dimensions is proved by a direct calculation in \cite{CHZ:17cr}. The upper bound in three dimension in proved in Lemma 3.3 of \cite{CHZ:17}. To show that $C h_F \|\jump{\nabla v_{cr}\times \bn_F}\|_{0,F}^2 \leq \dfrac{1}{h_F} \|\jump{v_{cr}}\|_{0,F}^2$ in a discrete setting, we let $ \|\jump{v_{cr}}\|_{0,F} =0$, then $v_{cr}$ is continuous on $F$, we get $\|\jump{\nabla v_{cr}\times \bn_F}\|_{0,F}=0$. By the discrete norm equivalence, we get $C h_F \|\jump{\gamma_{t_F}(\nabla v_{cr})}\|_{0,F}^2 \leq \dfrac{1}{h_F} \|\jump{v_{cr}}\|_{0,F}^2$. The weight can be obtained by the standard reference argument.

For $v_{cr}\in V_D^{cr}$ and $\btau_{rt}\in RT_{0,N}$, with the property of the CR elements, we have
\begin{eqnarray*}
(\nabla_h v_{cr}, \btau_{rt}) &=& -(v_{cr}, \gradt \btau_{rt})+ \sum_{K\in\cT}\sum_{F\in \p K}(\btau_{rt}\cdot\bn_F, v_{cr})_F\\
&=&-(v_{cr}, \gradt \btau_{rt})+\sum_{F\in \cE_{int}\cup \cE_D}(\btau_{rt}\cdot\bn_F, \jump{v_{cr}})_F
= -(v_{cr}, \gradt \btau_{rt}).
\end{eqnarray*}
Thus we have the following discrete version of integration by parts,
\beq \label{discrete_ibp}
(\nabla_h v_{cr}, \btau_{rt}) + (v_{cr}, \gradt\btau_{rt}) =0 \quad\forall\; v_{cr}\in V_D^{cr} \mbox{ and }\btau_{rt}\in RT_{0,N}
\eeq

\subsection{Helmholtz Decompositions}
For simplicity, we assume both $\Gamma_D$ and $\Gamma_N$ are not empty.

The two-dimensional version of the Helmholtz decomposition can be found in \cite{GR:86,CZ:10a}.
\begin{lem}\label{hm2D}
For a vector-valued function $\btau \in L^2(\O)^2$, there exists $\a\in H^1_D(\O)$ and $\beta\in H^1_N(\O)$ such that 
$$
\btau = A\nabla \alpha + \gperp \beta.
$$
\end{lem}

We have the following Helmholtz decomposition in three dimensions (Theorem 2.6 \cite{CMM:97} and Theorem 2.1 of \cite{CBJ:02}).
\begin{lem}\label{hm3D}
Assume that $\O$ is a simply connected,  bounded, and open domain in $\R^3$. The boundary $\p\O$ is Lipschitz continuous. Then, for a vector-valued function $\btau \in L^2(\O)^3$, there exists a unique $\a\in H^1_D(\O)$ and a unique $\beta\in H^1(\O)^3$ with 
$\gradt \bbeta = 0$, $\bbeta\cdot\bn=0$ on $\Gamma_D$, $(\nabla \times \bbeta)\cdot\bn=0$ on $\Gamma_N$ and $\bbeta\times\bn=0$ on $\Gamma_N$, such that 
$$
\btau = A\nabla \alpha + \nabla \times \bbeta 
\quad
\mbox{and}\quad 
\|\bbeta\|_1 \leq C \|\nabla \times \bbeta\|_0.
$$
\end{lem}
\begin{rem}
A discussion of the more complicated multiply-connected domain can also be found in \cite{CMM:97}; we omit it in this paper for simplicity.
\end{rem}

Let
\beq
\Q = \left\{\begin{array}{lll}
H^1_N(\O)
& d=2,\\[2mm]
 \{\bbeta \in H^1(\O)^3: \gradt\bbeta =0,  \bbeta\cdot\bn=0 \mbox{ on }\Gamma_D, 
 (\nabla \times \bbeta)\cdot\bn=0 \mbox{ and } \bbeta\times\bn=0 \mbox{ on }\Gamma_N\} & d=3.
\end{array}
\right.
\eeq
We have the following Poincar\'{e}-Friedrichs inequality,
\beq\label{PF}
\|\tq\|_0 \leq C \|\tcurl \tq\|_0 \quad \forall \tq\in \Q.
\eeq
Thus $ \|\tcurl \tq\|_0$ is a norm on $\Q$.

The following integrations by parts formulas also hold:
\begin{eqnarray}
\label{ip1}
(\nabla p, \tcurl \tq) &=& 0 \quad \forall (p, \tq)\in H_D^1(\O)\times \Q,\\
\label{ip2}
(\btau, \tcurl \tq) - (\nabla \times \btau, \tq) &=& 0 \quad \forall (\btau, \tq)\in H_D(\curll;\O)\times \Q.
\end{eqnarray}
The identities in  \eqref{ip1} and the two-dimensional case in  \eqref{ip2} are easy to prove. 
We only give proof of the second identity in three dimensions. First, we have the following integration by parts formula for $H(\curll)$ functions (see Theorem 3.31 of \cite{Monk:06}), 
$$
(\nabla \times \btau, \tq) - (\btau, \nabla \times \tq)  = (\gamma_t(\btau), \gamma_T(\tq))_{\p\O}, \quad \forall (\btau,\tq) \in H(\curll;\O)^2,
$$
where $ \gamma_T(\tq) = (\tq\times \bn)\times\bn$.
Thus, for $(\btau, \tq)\in H_D(\curll;\O)\times \Q$, we have
\begin{eqnarray*}
(\nabla \times \btau, \tq) - (\btau, \nabla \times \tq)  &=&
(\gamma_t(\btau), \gamma_T(\tq))_{\p\O} =  (\btau\times \bn, (\tq\times \bn)\times\bn))_{\p\O} \\
& =&  (\btau\times \bn, (\tq\times \bn)\times\bn)_{\G_D}+ (\btau\times \bn, (\tq\times \bn)\times\bn))_{\G_N} =0.
 \end{eqnarray*}
In summary, we have the following Helmholtz decomposition theorem
\begin{thm}\label{hd}
Assuming the conditions of Lemmas \ref{hm2D} and \ref{hm3D} hold, we have: For
$\btau \in L^2(\O)^d$, there exists $\a\in H^1_D(\O)$ and $\beta\in \Q$ such that 
\beq
\btau = A\nabla \alpha + \tcurl \tbeta \mbox{ and }
\|A^{-1/2}\btau\|_0^2 = \|A^{1/2}\nabla \alpha\|_0^2 + \|A^{-1/2} \tcurl  \tbeta\|_0^2.
\eeq
\end{thm}
The identity $\|A^{-1/2}\btau\|_0^2 = \|A^{1/2}\nabla \alpha\|_0^2 + \|A^{-1/2} \tcurl  \tbeta\|_0^2$ is a simple consequence of the orthogonality relation \eqref{ip1}.

\subsection{Discrete Poincar\'{e}-Friedrichs Inequality}
The proof of the following discrete Poincar\'{e}-Friedrichs inequality can be found in papers \cite{Bre:03,Voh:05}.
\begin{thm}\label{disPoin}
There exists a positive constant $C$ depending on $\O$ such that 
\beq \label{dPF}
\|v\|_0 \leq C \|\nabla_h v\|_0 \quad v\in W^{1}_D(\cT).
\eeq
\end{thm}

With the help of the discrete Poincar\'{e}-Friedrichs inequality, we can  define the following combined norm for the space $H_N(\divvr;\O)\times W^{1}_D(\cT)$:
$$
\tri (\btau, v) \tri ^2 := \|\nabla_h v\|^2_0 + \|\btau\|_0^2+ \|\gradt \btau\|_0^2 
\quad (\btau, v) \in H_N(\divvr;\O)\times W^{1}_D(\cT).
$$

%

\section{Two-Field Potential-Flux Div CR-LSFEM}
\setcounter{equation}{0}

This section introduces the two-field potential-flux div least-squares method for general second-order elliptic equations with Crouzeix-Raviart elements.

For $(\btau,v) \in H_N(\divvr;\O)\times W_D^{1+cr}$, define the two-field potential-flux div least-squares functional for \eqref{pde1} as:
\beq
\cJ^{div}_{h}(\btau,v;f) := \|A^{1/2}\nabla_h v+ A^{-1/2}\btau\|_0^2 + \|\gradt \btau + X_h v -f\|_0^2,
\eeq
where 
$X_h v= \bb\cdot\nabla_h v + cv$.

For $(\bchi,w) \in H_N(\divvr;\O)\times W_D^{1+cr}$ and $(\btau,v) \in  H_N(\divvr;\O)\times W_D^{1+cr}$, define the following  bilinear form $b_h$:
\begin{eqnarray}
b_h((\bchi,w), (\btau,v))&:=& (A\nabla_h w+ \bchi, \nabla_h v+ A^{-1}\btau)+(\gradt\bchi+X_h w, \gradt\btau+X_h v).
\end{eqnarray}
When $v\in H_D^1(\O)$, we can remove the subscript $h$ in $\nabla_h$ of the above definitions, and they are the standard potential-flux div least-squares functional and bilinear forms.
We have,
$$
\cJ^{div}_{h}(\btau,v;0) = b_h((\btau,v), (\btau,v)) \quad (\btau,v) \in  H_N(\divvr;\O)\times W_D^{1+cr}.
$$
The solution $(\bsigma, u)$ satisfies the least-squares minimization problem:
Find $(\bsigma,u)\in H_N(\divvr;\O)\times H^1_D(\O)$ such that
\beq \label{lsfem-divl2_min}
\cJ^{div}_{h}(\bsigma,u;f)
= \inf_{(\btau,v)\in H_N(\divvr;\O)\times H^1_D(\O)}\cJ^{div}_{h}(\btau,v; f).
\eeq
Or equivalently, a weak problem: Find $(\bsigma,u)\in H_N(\divvr;\O)\times H^1_D(\O)$ such that
\beq \label{lsfem-divl2}
b_{h}((\bsigma,u), (\btau,v)) = (f, \gradt\btau+X_h v) \quad \forall (\btau,v)\in H_N(\divvr;\O)\times H^1_D(\O).
\eeq
The following norm equivalence is standard.  Different proofs can be found in \cite{CLMM:94,Cai:04,CFZ:15,Ku:07,Zhang:22}.
\begin{thm} Assuming one of the conditions in Theorem \ref{assp} is true, then for all $(\btau,v)\in H_N(\divvr;\O)\times H_D^1(\O)$, we have
\beq \label{norm-equ-standard}
C_1 \tri (\btau,v) \tri^2 \leq \cJ^{div}_{h}(\btau,v;0) =\|A^{1/2}\nabla v+ A^{-1/2}\btau\|_0^2 + \|\gradt \btau + X v\|_0^2
\leq C_2 \tri (\btau,v) \tri^2.
\eeq
\end{thm}
Consider the two-field potential-flux div CR-LSFEM:
Find $(\bsigma_{rt},u_{cr})\in RT_{0,N}\times V^{cr}_D$ such that
\beq
\cJ^{div}_{h}(\bsigma_{rt},u_{cr};f)
= \inf_{(\btau_{rt},v_{cr})\in RT_{0,N}\times V^{cr}_D}\cJ^{div}_{h}(\btau_{rt},v_{cr}; f).
\eeq
Or equivalently: Find $(\bsigma_{rt},u_{cr})\in RT_{0,N}\times V^{cr}_D$ such that
\beq \label{cr-lsfem-divl2}
b_{h}((\bsigma_{rt},u_{cr} ), (\btau_{rt},v_{cr})) = F_h(\btau_{rt},v_{cr}) \quad \forall (\btau_{rt},v_{cr})\in RT_{0,N}\times V^{cr}_D.
\eeq
Unlike the standard CR finite element method \cite{CR:73,Brenner:15}, as discussed in \cite{DL:03}, we have the error equation for the least-squares method.
\begin{lem} Assume that $(\bsigma,u)\in H_N(\divvr;\O)\times H^1_D(\O)$ is the solution of  \eqref{lsfem-divl2} and $(\bsigma_{rt},u_{cr})\in RT_{0,N}\times V^{cr}_D$  is the numerical solution of \eqref{cr-lsfem-divl2}.
The following error equation is true,
\beq\label{error_eq}
b_h((\bsigma-\bsigma_{rt},u-u_{cr}), (\btau_{rt},v_{cr})) = 0 \quad \forall (\btau_{rt},v_{cr})\in RT_{0,N}\times V^{cr}_D.
\eeq
\end{lem}
\begin{proof}
Since $\bsigma$ and $u$ satisfy the first-order system \eqref{1os} in the $L^2$ sense, then for any $(\btau_{rt},v_{cr})\in RT_{0,N}\times V^{cr}_D$, we have
\begin{eqnarray*}
b_h((\bsigma,u), (\btau_{rt},v_{cr}))& =& (A\nabla u+ \bsigma, \nabla_h v_{cr}+ A^{-1}\btau_{rt})+(\gradt\bsigma+X u, \gradt\btau_{rt}+X_h v_{cr}) \\
&=& (f, \gradt\btau_{rt}+X_h v_{cr}) = b_h((\bsigma_{rt},u_{cr}), (\btau_{rt},v_{cr})).
\end{eqnarray*}
Thus the error equation is proved.
\end{proof}

\section{Discrete Existence and Uniqueness of the Two-Field Div CR-LSFEM: I. Coerciveness Assumption}
\setcounter{equation}{0}
We first present a coerciveness result of the two-field potential-flux div CR-
LSFEM with some assumptions of the coefficients ensuring the coerciveness of the original bilinear form $a(\cdot,\cdot)$ \eqref{a}.

\begin{assumption}\label{asmp_coef}
Assume that Assumption \ref{ass_dc} is true. It is further assumed that $\gradt\bb\in L^{\infty}(\O)$ and  $c\in L^{\infty}(\O)$ satisfy the following conditions:
$c-\frac12\gradt\bb\geq 0$ in $\O$ and $\bn\cdot\bb\geq 0$ on $\Gamma_N$.
\end{assumption}

\begin{thm}
Assuming that Assumption \eqref{asmp_coef} is true and the mesh size of $\cT$ is small enough, there exists constants $C_1 > 0$ and $C_2>0$ independent of mesh size, such that
\beq \label{coer1}
C_1\tri (\btau_{rt}, v_{cr}) \tri ^2 \leq \cJ^{div}_{h}(\btau_{rt}, v_{cr};0)
\leq C_2\tri (\btau_{rt}, v_{cr}) \tri ^2 \quad \forall (\btau_{rt},v_{cr})\in RT_{0,N}\times V^{cr}_D,
\eeq
\end{thm}
\begin{proof} 
The upper bound is easy to prove by the triangle and the discrete Poincar\'e inequalities. Our main task is to prove the coerciveness.

Using integration by parts, we can easily derive that
$
(\bb\cdot\nabla_h v_{cr},v_{cr}) = -\frac12(\gradt\bb,v_{cr}^2) + \frac12\sum_{K\in\cT}(\bn\cdot\bb , v_{cr}^2)_{\p K}.
$
Then due to fact that $\bn\cdot\bb\geq 0$ on $\Gamma_N$ and the uniform bound of $c-\frac12\gradt\bb\geq 0$ in $\O,$
\begin{eqnarray*}
 -(X_h v_{cr},v_{cr}) &=& -(c-\frac12\gradt\bb,v_{cr}^2) -\frac12\sum_{K\in\cT}(\bn\cdot\bb , v_{cr}^2)_{\p K} \leq - \frac12\sum_{K\in\cT}(\bn\cdot\bb , v_{cr}^2)_{\p K}
\\
&\leq &- \frac12\sum_{F\in\cE_{int}}(\bn\cdot\bb, \jump{v_{cr}^2})_F
- \frac12\sum_{F\in\cE_D}(\bn\cdot\bb, v_{cr}^2)_F.
\end{eqnarray*}
For any $v_{cr}\in V^{cr}$ and any $F\in\cE_{int}$, by the fact that $\int_F \jump{v_{cr}} ds =0$, the mean value of $v_{cr}$ over $F$ is single-valued constant, 
$
\overline{v}_{cr,F} = \Pi_F^0  v_{cr}|_{K_F^+} = \Pi_F^0  v_{cr}|_{K_F^-}, 
$
where $K_F^+$ and $K_F^-$ are two elements sharing the common $F$. We also have $\overline{v}_{cr,F} =0$ for $F\in\cE_D$.  Then for an $F\in\cE_{int}$, we have
\begin{eqnarray*}
\frac{1}{2}(\bb\cdot\bn, \jump{v_{cr}^2})_F &=& (\bb\cdot\bn \jump{v_{cr}}, \{v_{cr}\})_F \leq 
|\bb\cdot\bn|_{\infty,F}|(\jump{v_{cr}}, \{v_{cr}\})_F| \\
&=& |\bb\cdot\bn|_{\infty,F}|(\jump{v_{cr}}, \{v_{cr}-\overline{v}_{cr,F}\})_F| \\
&\leq& |\bb\cdot\bn|_{\infty,F} \| \jump{v_{cr}}\|_{F}  (\|v_{cr}|_{K_F^-}  - \overline{v}_{cr,F}\|_{0,F}+\|v_{cr}|_{K_F^+} - \overline{v}_{cr,F}\|_{0,F})/2.
\end{eqnarray*}
We have (see e.g. p.110 of \cite{Braess:07})
$$
 \|v_{cr}|_{K_F} - \overline{v}_{cr,F}\|_{0,F} \leq C h_F^{1/2} \|\nabla v_{cr}\|_{0,K_F}.
$$
Combing the above estimate and \eqref{ejump}, we get 
\begin{eqnarray*}
(\bb\cdot\bn, \jump{v_{cr}^2})_F \leq C_F h_F|\bb\cdot\bn|_{\infty,F} \|\nabla v_{cr}\|_{0, K_F^-\cup K_F^+}^2 \quad \mbox{ for } F\in\cE_{int},
\end{eqnarray*}
where $C_F$ only depends on the shape of $K\in K_F^-\cup K_F^+$.
Similarly,   we have 
$$
(\bb\cdot\bn, v_{cr}^2)_F \leq C_F h_F|\bb\cdot\bn|_{\infty,F} \|\nabla v_{cr}\|_{0, K_F}^2
\quad \mbox{ for } F\in\cE_{D}.
$$
Thus, we have
\beq \label{Xvv}
-(X_h v_{cr},v_{cr}) \leq C\sum_{K\in\cT}(\sup_{F\in \p K \cap(\cE_{int} \cup \cE_D)} |\bb\cdot\bn|_{\infty,F}h_F )\|\nabla v_{cr}\|^2_{0,K}.
\eeq
By \eqref{discrete_ibp}, the discrete Poincar\'{e} inequality, and \eqref{Xvv},
\begin{eqnarray*}
\|A^{1/2}\nabla_hv_{cr}\|^2_0
&=&(A^{1/2}\nabla_hv_{cr}+A^{-1/2}\btau_{rt},A^{1/2}\nabla_hv_{cr})
+(\gradt\btau_{rt}+X_h v_{cr},v_{cr}) -(X_h v_{cr},v_{cr})\\
&\leq& C\|A^{1/2}\nabla_hv_{cr}+A^{-1/2}\btau_{rt}\|_0\|\nabla_h v_{cr}\|_0
+C\|\gradt\btau_{rt}+X_h v_{cr}\|_0\|\nabla_h v_{cr}\|_0\\
&&+C_0\sum_{K\in\cT}(\sup_{F\in \p K \cap(\cE_{int} \cup \cE_D)} |\bb\cdot\bn|_{\infty,F}h_F )\|\nabla v_{cr}\|^2_{0,K}.
\end{eqnarray*}
where $C_0$ only depends on the shape regularity of the mesh $\cT$. Choosing the mesh size small enough such that on each element $K$, 
\beq \label{chose_h}
C_0\sup_{F\in \p K \cap(\cE_{int} \cup\cE_{D})}  |\bn_F\cdot\bb|_{\infty,F}h_F<\frac12\Lambda_0|_K,
\eeq
we have the following bound for a $C>0$ independent of the mesh size,
\beq 
C\|\nabla_hv_{cr}\|^2_0 \leq \cJ^{div}_{h}(\btau_{rt},v_{cr};0).
\eeq
An application of the triangle inequality shows that
\beq \label{tmp1}
C\|\btau_{rt}\|_0 \leq \|A^{1/2}\nabla_hv_{cr}+A^{-1/2}\btau_{rt}\|_0+ C \|\nabla_hv_{cr}\|_0, 
\eeq
\beq \label{tmp2}
\mbox{and  } C\|\gradt \btau_{rt}\|_0 \leq \|\gradt\btau_{rt}+X_h v_{cr}\|_0 + \|X_hv_{cr}\|_0 \leq  
\|\gradt\btau_{rt}+X_h v_{cr}\|_0+ C \|\nabla_hv_{cr}\|_0.
\eeq
The lemma is then proved.
\end{proof}
\begin{rem}
We know from \eqref{chose_h} that the mesh size needs not to be extremely small. It depends on the shape of the element, the coefficient $A$ on the element, and $\bb\cdot\bn_F$ on its faces/edges. Also, \eqref{chose_h} is a local result suitable for an adaptively refined non-uniform mesh. 

No regularity assumption is needed in this proof.

Assumption \ref{asmp_coef} is not the most general one to ensure the coerciveness of the bilinear form $a$. One can find in \cite{QV:94} for a more general setting, and the analysis presented in this section can also be extended to the more general setting.  The proof in this section shows that an explicit assumption on the coefficient helps determine the local mesh size to ensure the coerciveness of the discrete problem.

We also notice that even with Assumption \ref{asmp_coef}, the original bilinear form for the pure CR finite element is only coercive under the assumption that the mesh is fine enough.
\end{rem}
With the coerciveness result \eqref{coer1}, the error equation \eqref{error_eq}, local approximation properties \eqref{rti}, \eqref{divrti}, and \eqref{localcr}, we immediately have the following a priori error estimate.
\begin{thm}For piecewise constant function $s$ on mesh $\cT$ with $s_K = s|_K$ satisfying $0<s_K \leq 1$, assume that $u|_K \in H^{1+s_K}(K)$, $\bsigma|_K \in H^{s_K}(K)$, and $\gradt\bsigma|_K \in H^{s_K}(K)$,  for $K\in \cT$. We also assume that Assumption \eqref{asmp_coef} is true and the mesh size of $\cT$ is small enough, then the following a priori error estimate is true:
\begin{eqnarray*} \label{apriori_ls}
\tri (\bsigma-\bsigma_{rt},u-u_{cr})\tri &\leq & C\inf_{(\btau_{rt},v_{rt})\in RT_{0,N}\times V^{cr}_D}\tri (\bsigma-\btau_{rt},u-v_{cr})\tri
\\
&\leq &C\sum_{K\in\cT}h_K^{s_K} (|u|_{1+s_K,K}+ |\bsigma|_{s_K,K} +  |\gradt\bsigma|_{s_K,K}).
\end{eqnarray*}
\end{thm}

\section{Discrete Existence and Uniqueness of the Two-Field Div CR-LSFEM: II. Schatz's Argument}
\setcounter{equation}{0}
In this section, we present a discrete existence and uniqueness proof based on the argument of Schatz \cite{Sch:74} without assuming Assumption \ref{asmp_coef} to ensure the coerciveness of the bilinear form $a$. Thus, the original variational problem \eqref{vp1} can be indefinite. 

In this proof, we only assume one of the conditions of Theorem \ref{assp} and a minimal regularity of the original and adjoint elliptic equations.
First, we prove an $L^2$-error estimate of the two-field potential-flux div method, and then a Garding-like inequality is proved. With these tools, we can get an a priori estimate, then the existence and uniqueness of the discrete problem follow.

\subsection{An $L^2$ estimate of the potential-flux div CR-LSFEM}
In this subsection, we present an $L^2$-error estimate of the two-field potential-flux div CR-LSFEM based on the argument of \cite{CK:06}. We also correct a minor mistake in the original proof of \cite{CK:06}.

We assume the following very mild regularity assumption.
\begin{assumption}\label{regularity}
Assume one of the conditions of Theorem \ref{assp} is true, and the following $H^{1+s}$ regularity estimates are true for \eqref{vp1} and \eqref{vp2}, respectively:
\beq\label{regularityeq}
\|u\|_{1+s} \leq C \|f\|_{0} \quad\mbox{and}\quad \|z\|_{1+s} \leq C \|g\|_{0}, \quad \mbox{for some }
0<s\leq 1.
\eeq
\end{assumption}
To discuss  an $L^2$-error estimate, we introduce a first-order system:
\begin{equation}\label{first_order_g}
\left\{
\begin{array}{llll}
A\grad w +   \bgamma  & =& A\nabla z & \mbox{in } \O
 \\[1mm]
\gradt\bgamma + Xw &=& z & \mbox{in } \O,
\end{array}
\right.
\end{equation}
with boundary conditions $
w=0 \mbox{ on }\Gamma_D$ and
$\bn\cdot\bgamma=0 \mbox{ on } \Gamma_N$.

\begin{lem}
Let $z$ be the solution of the adjoint problem \eqref{vp2} and $(\bgamma,w)$ be the functions defined in \eqref{first_order_g}, we have the following estimates under Assumption \ref{regularity}:
\beq\label{bdd_g}
\|w\|_{1+s}+ \|\bgamma\|_s +\|\gradt \bgamma\|_s\leq C\|g\|_0.
\eeq
\end{lem}
\begin{proof}
The first-order system \eqref{first_order_g} can be understood in the PDE form as:
$$
-\gradt(A\grad w) + \bb\cdot\grad w + cw = z - \gradt(A\nabla z) \mbox{ in } \O,
$$
with boundary conditions 
$$
w=0 \mbox{ on } \Gamma_D\quad\mbox{and}\quad A\nabla w\cdot\bn =  A\nabla z\cdot\bn = (\bb\cdot \bn)z \mbox{  on }\Gamma_N. 
$$
Note that the boundary condition of $w$ on $\Gamma_N$ is not standard. Thus the claim in p.1729 (below (5.7)) of the paper \cite{CK:06} is not accurate. The boundary condition shows that $w$ is not a solution of $w\in H_D^1(\O)$ with
$
a(w,v) = (z - \gradt(A\nabla z), v)$, for all $v\in H_D^1(\O)$. 
We can not use the regularity result of \eqref{vp1} directly. To fix this, introduce 
$
y = w-z \in H_D^1(\O).
$ 
Then 
$$
\left\{
\begin{array}{llll}
A\nabla y  +   \bgamma  & =& 0 & \mbox{in } \O
 \\[1mm]
\gradt\bgamma + X y &=& z - X z& \mbox{in } \O.
\end{array}
\right.
$$
Then we have $y \in H_D^1(\O)$ satisfying 
$$
-\gradt(A\nabla y) + X y = z - X z \quad \mbox{in } \O,
\quad
 y=0\mbox{ on } \Gamma_D,\quad A\nabla y \cdot\bn =0 \mbox{  on }\Gamma_N.
$$
Or, equivalently,
$$
a(y,v) = (z - Xz, v) \quad \forall v\in H_D^1(\O).
$$
Then by the regularity assumption \eqref{regularityeq} (first for $y$, then for $z$), we have
$$
\|y\|_{1+s} \leq C \|z - X z\|_{0} \leq C \|z\|_{1} \leq  C\|g\|_0.
$$
Then, the following estimate for $w$ is true:
$$
\|w\|_{1+s} = \|y+z\|_{1+s} \leq  \|y\|_{1+s}+\|z\|_{1+s} \leq C\|g\|_0.
$$
We also have
\begin{eqnarray*}
\|\bgamma\|_s &=& \|A\nabla y\|_s \leq C\|y\|_{1+s} \leq C\|g\|_0,\\
\|\gradt \bgamma\|_s &=& \|z-Xw\|_s \leq C(\|z\|_1 + \|w\|_{1+s}) \leq C\|g\|_0.
\end{eqnarray*}
This completes the proof of the lemma.
\end{proof}
\begin{lem}\label{estjump}
 Assume $h$ is the maximum mesh size of the mesh $\cT$. Let $z$ be the solution of \eqref{vp2}, $u$ be the solution of \eqref{pde1},  and $u_{cr}$ be the solution of \eqref{cr-lsfem-divl2}, then the following inequality holds under  Assumption \ref{regularity}:
\beq
 \sum_{F\in\cE_{int}\cup\cE_D}((A\nabla z+\bb z)\cdot\bn, \jump{u_{cr}})_{F} \leq C h^s \|\nabla_h (u-u_{cr})\|_0 \|g\|_0.
\eeq
\end{lem}
\begin{proof}
Let $\bxi = A\nabla z+\bb z$, then $\gradt \bxi = cz-g \in L^2(\O)$, thus $\bxi \in H_N(\divvr;\O)$. By the regularity assumption on $z$, we have $\bxi \in H^s(\O)$. Define $\bxi_{rt0}$ to be $\bxi$'s interpolation in $RT_{0,N}$. We have 
$$
\|\bxi-\bxi_{rt0}\|_{0} \leq C h^s \|\bxi\|_{s}\quad\mbox{and}\quad
\|\gradt \bxi - \gradt \bxi_{rt0}\|_0 \leq \|\gradt \bxi\|_0 =\|g- cz\|_0 \leq C \|g\|_0.
$$ 
Summing up all elements and using the regularity assumption and the fact $0<s\leq 1$, we get
\beq \label{xierror}
\sum_{K\in\cT} (\|\bxi-\bxi_{rt0}\|_{0,K} +h_K\|\gradt (\bxi-\bxi_{rt0})\|_{0,K})\leq h^s \|g\|_0.
\eeq

By the property of $(\jump{u_{cr}}, 1)_F=0$ and the fact $\bxi_{rt0}\cdot\bn|_F$ is a constant, we have 
\beq
(\bxi_{rt0}\cdot\bn,\jump{u_{cr}})_F=0 \quad \forall F\in \cE_{int}\cup \cE_D.
\eeq
By the above result, the trace theorem of Lemma 2.4 and Remark 2.5 of \cite{CHZ:17}, \eqref{jumperror}, and \eqref{xierror}, we have
\begin{eqnarray*}
&&\sum_{F\in\cE_{int}\cup \cE_D}(\bxi\cdot\bn, \jump{u_{cr}})_{F} =
\sum_{F\in\cE_{int}\cup \cE_D}((\bxi-\bxi_{rt0})\cdot\bn, \jump{u_{cr}})_{F} \\
&\leq& C \sum_{F\in\cE_{int}\cup \cE_D} h_F^{-1/2}\|\jump{u_{cr}}\|_{0,F} (\|\bxi-\bxi_{rt0}\|_{0,K_F^-\cup K_F^+} + h_K \|\gradt (\bxi-\bxi_{rt0})\|_{0,K_F^-\cup K_F^+}) \\
&\leq& C h^s \|\nabla_h (u-u_{cr})\|_0 \|g\|_0.
\end{eqnarray*}
\end{proof}
We also have the following approximation property by \eqref{localcr}, \eqref{rti}, \eqref{divrti}, and \eqref{bdd_g},
\beq \label{appp}
\inf_{(\bgamma_h,w_h)\in RT_{0,N}\times V_D^{cr}}\tri (\bgamma-\bgamma_h,w-w_h ) \tri
\leq C h^s (\|\bgamma\|_s + \|\gradt \bsigma\|_s + \|w\|_{1+s}) \leq Ch^s\|g\|_0.
\eeq

\begin{thm}
Assume $h$ is the maximum mesh size of the mesh $\cT$. 
Let $(\bsigma,u)$ be the solution of \eqref{1os}  and $(\bsigma_{rt},u_{cr})$ be the solution of \eqref{cr-lsfem-divl2}, then the following $L^2$-error estimate holds under Assumption \ref{regularity}:
\beq\label{L2}
\|u-u_{cr}\|_0 \leq C h^s \tri (\bsigma-\bsigma_{rt},u-u_{cr}) \tri.
\eeq
\end{thm}
\begin{proof}
Let $e= u-u_{cr} \in W_D^{1+cr}$ and $\bE= \bsigma-\bsigma_{rt}\in H_N(\divvr;\O)$.
Multiplying both sides of \eqref{pde2} by $e$ and integrating by parts,  we have
\begin{eqnarray*}
(e,g)&=& -(\gradt(A\nabla z+\bb z), e) + (ce, z) \\
&=& (A\nabla z+\bb z,\nabla_h e)+(cz, e) - \sum_{K\in\cT}((A\nabla z+\bb z)\cdot\bn_{\p K}, e)_{\p K}\\
&=& (A\nabla_h e, \nabla z)+(X_h e,z) + \sum_{F\in\cE_{int}\cup\cE_D}((A\nabla z+\bb z)\cdot\bn, \jump{u_{cr}})_{F}.
\end{eqnarray*}
Using the fact that $(\bE,\nabla z)+ (\gradt \bE,z)=0$, the system \eqref{first_order_g}, the error equation \eqref{error_eq}, and  the approximation property \eqref{appp}, the first two terms can be bounded by the following estimate,
\begin{eqnarray*}
(A\nabla_h e, \nabla z)+(X_h e,z)&=& (A\nabla_h e +\bE, \nabla z)+(\gradt \bE+X_h e,z)  \\
&=& (A\nabla_h e +\bE, \nabla w+A^{-1}\bgamma)+(\gradt \bE+X_h e,\gradt \bgamma+Xw)   \\
&=& b_h(\bE,e;\bgamma,w)  =  b_h(\bE,e;\bgamma-\bgamma_h,w-w_h)\\
&\leq& \tri(\bE,e)\tri \tri (\bgamma-\bgamma_h,w-w_h) \tri \leq ch^s \|g\|_0  \tri(\bE,e)\tri 
\end{eqnarray*}
Combining with Lemma \ref{estjump}, we have the result of the theorem.
\end{proof}

\subsection{Discrete existence and uniqueness based on Schatz's argument}
We first prove a discrete Garding-like inequality, then derive a priori error estimate, and show the existence and uniqueness of the discrete potential-flux div CR method by using Schatz's argument \cite{Sch:74}.
\begin{lem}\label{Garding}
The following inequality is true for $ (\btau_{rt},v_{cr})\in RT_{0,N}\times V^{cr}_D$:
\beq
C\tri (\btau_{rt}, v_{cr})\tri^2 \leq \cJ^{div}_{h}(\btau_{rt},v_{cr};0) + \|v_{cr}\|_0^2
= b_h((\btau_{rt}, v_{cr}),(\btau_{rt}, v_{cr})) + \|v_{cr}\|_0^2.
\eeq
\end{lem}
\begin{proof}
By \eqref{discrete_ibp} and the discrete Poincar\'{e} inequality,
\begin{eqnarray*}
\|A^{1/2}\nabla_hv_{cr}\|^2_0
&=&(A^{1/2}\nabla_hv_{cr}+A^{-1/2}\btau_{rt},A^{1/2}\nabla_hv_{cr})
+(\gradt\btau_{rt}+X_h v_{cr},v_{cr}) -(X_h v_{cr},v_{cr})\\
&\leq& C(\|A^{1/2}\nabla_hv_{cr}+A^{-1/2}\btau_{rt}\|_0
+\|\gradt\btau_{rt}+X_h v_{cr}\|_0)\|\nabla_h v_{cr}\|_0  +C\|\nabla_h v_{cr}\|_0 \|v_{cr}\|_0\\
&\leq& C(\|A^{1/2}\nabla_h v_{cr}+A^{-1/2}\btau_{rt}\|_0
+\|\gradt\btau_{rt}+X_h v_{cr}\|_0 + \|v_{cr}\|_0)\|A^{1/2}\nabla_hv_{cr}\|_0.
\end{eqnarray*}
Thus, 
$$
\|\nabla_hv_{cr}\|_0 \leq C(\|A^{1/2}\nabla_hv_{cr}+A^{-1/2}\btau_{rt}\|_0
+\|\gradt\btau_{rt}+X_hv_{cr}\|_0 + \|v_{cr}\|_0).
$$
The lemma is proved by combining the above result, \eqref{tmp1}, and \eqref{tmp2}.
\end{proof}

\begin{thm}
Assume Assumption \ref{regularity} is true.
There exists an $h_0>0$, such that when the maximum mesh size $h$ of $\cT$ is smaller than $h_0$, the discrete problem of the two-field potential-flux div CR-LSFEM \eqref{lsfem-divl2_min} or \eqref{lsfem-divl2} has a unique solution.
For piecewise constant function $s$ on mesh $\cT$ with $s_K = s|_K$ satisfying $0<s_K \leq 1$, assume that $u|_K \in H^{1+s_K}(K)$, $\bsigma|_K \in H^{s_K}(K)$, and $\gradt\bsigma|_K \in H^{s_K}(K)$,  for $K\in \cT$. The following a priori error estimate is true,
\begin{eqnarray*}
\tri (\bsigma-\bsigma_{rt}, u-u_{cr})\tri &\leq& C\inf_{(\btau_{rt},v_{rt})\in RT_{0,N}\times V^{cr}_D}\tri (\bsigma-\btau_{rt},u-v_{cr})\tri\\
&\leq &C\sum_{K\in\cT}h_K^{s_K} (|u|_{1+s_K,K}+ |\bsigma|_{s_K,K} +  |\gradt\bsigma|_{s_K,K}).
\end{eqnarray*}
\end{thm}
\begin{proof}
Let $v_{cr}$ be an arbitrary function in $V^{cr}_D$ and $\btau_{cr}$ be an arbitrary function in $RT_{0,N}$. Denote $e_h =  u_{cr}-v_{cr}$ and $\bE_h = \bsigma_{cr}-\btau_{cr}$. By the result of Lemma \ref{Garding}, the error equation \eqref{error_eq}, we have
\begin{eqnarray*}
\tri (\bE_h, e_h) \tri^2 &\leq & C(b_h(\bE_h, e_h;\bE_h, e_h) +\|e_h\|^2_0)\\
&\leq & C(b_h(\bsigma-\btau_{rt}, u-v_{cr};\bE_h, e_h) +\|u-v_{cr}\|^2_0 + \|e\|_0^2)\\
&\leq&C(\tri(\bsigma-\btau_{rt}, u-v_{cr})\tri  \tri(\bE_h, e_h)\tri + \tri(\bsigma-\btau_{rt}, u-v_{cr})\tri ^2 +\|e\|^2_0)
\end{eqnarray*}
An application of  Young's inequality with $\epsilon$ shows
\beq \label{Eheh}
C\tri (\bE_h, e_h) \tri \leq \tri(\bsigma-\btau_{rt}, u-v_{cr})\tri  +\|e\|_0.
\eeq
Then by the triangle inequality, \eqref{Eheh}, and the $L^2$-error estimates \eqref{L2},
\begin{eqnarray*}
\tri(\bE,e)\tri &\leq& \tri(\bE_h,e_h)\tri +\tri(\bsigma-\btau_{cr},u-v_{cr})\tri
\leq C(\tri(\bsigma-\btau_{cr},u-v_{cr})\tri +\|e\|_0)\\
&\leq&C(\tri(\bsigma-\btau_{cr},u-v_{cr})\tri + h^s\tri(\bE,e)\tri).
\end{eqnarray*}
Choosing $h_0$ small enough, we have 
$$
\tri (\bsigma-\bsigma_{rt}, u-u_{cr})\tri  \leq  C\inf_{(\btau_{rt},v_{rt})\in RT_{0,N}\times V^{cr}_D}\tri (\bsigma-\btau_{rt},u-v_{cr})\tri.
$$
The other part of the a priori error analysis is from the local approximation properties \eqref{rti}, \eqref{divrti}, and \eqref{localcr}. 
The existence and uniqueness of the discrete problem are then a simple consequence of the a priori error estimate.
\end{proof}

\section{Discrete Existence and Uniqueness of the Two-Field Div CR-LSFEM: III. Inf-Sup Condition of A Standard CR Method}
\setcounter{equation}{0}
In this section, we present a third proof of the coerciveness of the two-field div CR-LSFEM by using the stability of the standard CR method for the second-order linear elliptic equation \cite{CDNP:16}.
Define the discrete bilinear form $a_h$ corresponding to \eqref{a},
\beq \label{ah}
a_h(w,v) := (A\nabla_h w,\nabla_h v) +(X_h w,v)\quad \mbox{for } w, v\in W^{1+cr}_D.
\eeq
Then the problem of the standard CR element approximation to  \eqref{pde1} is: Find $u_{cr}^p\in V^{cr}_D$, such that
\beq\label{CRp}
	a_h(u_{cr}^p,v_{cr}) = (f,v_{cr}) \quad \forall v_{cr} \in  V^{cr}_D.
\eeq
And the problem of the standard CR element approximation to the adjoint equation \eqref{pde2} is: Find $u_{cr}^d\in V^{cr}_D$, such that
\beq \label{CRd}
	a_h(v_{cr}, u^d_{cr}) = (f,v_{cr}) \quad \forall v_{cr} \in  V^{cr}_D.
\eeq
In \cite{CDNP:16}, for the problem \eqref{CRd} with $\Gamma_D = \p\O$ and a more general righthand side, the following CR approximation of the adjoint problem is considered: For all $f_0\in L^2(\O)$ and $\bff_1\in L^2(\O)^d$, find $u_{cr}^d\in V^{cr}_D$, such that   
\beq \label{cr_g}
	a_h(v_{cr}, u^d_{cr}) = (f_0,v_{cr}) + (\bff_1, \nabla_h v_{cr}) \quad \forall v_{cr} \in  V^{cr}_D.
\eeq
It is proved that \eqref{cr_g} has a unique solution if the mesh size is small enough and Assumption \ref{regularityeq} is true. The following stability result is also proved:
\beq \label{cr_stab}
\beta\|\nabla_h u^d_{cr}\|_0 \leq \|f_0\|_0 +\|\bff_1\|_0 \quad \forall f_0\in L^2(\O)\quad \bff_1\in L^2(\O)^d.
\eeq
It is easy to derive that the above result still holds for mixed boundary conditions with $\Gamma_D\neq \emptyset$.
By the \BNB theory \cite{Babuska:71,Braess:07,XZ:03}, it is known that the well-posedness of  \eqref{cr_g} is equivalent to the inf-sup stability of $a_h$:
\beq \label{infsup}
0<\beta = \inf_{v_{cr}\in V^{cr}_D}\sup_{w_{cr}\in V^{cr}_D} \dfrac{a_h(v_{cr},w_{cr})}{\|\nabla_h v_{cr}\|_0\|\nabla_h w_{cr}\|_0} = \inf_{w_{cr}\in V^{cr}_D}\sup_{v_{cr}\in V^{cr}_D} \dfrac{a_h(v_{cr},w_{cr})}{\|\nabla_h v_{cr}\|_0\|\nabla_h w_{cr}\|_0}. 
\eeq
The inf-sup condition \eqref{infsup} is also equivalent to
\beq\label{is}
\beta \|\nabla_h v_h\|_0 \leq \sup_{w_h\in V^{cr}_D} \dfrac{a_h(v_h,w_h)}{\|\nabla_h w_h\|_0}
\quad\mbox{and}\quad
\beta \|\nabla_h v_h\|_0 \leq \sup_{w_h\in V^{cr}_D} \dfrac{a_h(w_h,v_h)}{\|\nabla_h w_h\|_0}
\quad \forall v_h\in V^{cr}_D.
\eeq
It is then obvious that once we have \eqref{infsup} or \eqref{is}, we also have the well-posedness of the equation \eqref{CRp}.

With the help of the inf-sup condition of $a_h(\cdot,\cdot)$ \eqref{is}, we can prove the coerciveness following the argument in \cite{Zhang:22}. 
\begin{thm}
Assume Assumption \ref{regularity} is true.
There exists an $h_0>0$, such that when the maximum mesh size $h$ of $\cT$ is smaller than $h_0$, the following inequality is true for $ (\btau_{rt},v_{cr})\in RT_{0,N}\times V^{cr}_D$:
\beq
C\tri (\btau_{rt}, v_{cr})\tri^2 \leq \cJ^{div}_{h}(\btau_{rt},v_{cr};0).
\eeq
\end{thm}
\begin{proof}
By the discrete integration by parts \eqref{discrete_ibp}, for  $\btau_{rt} \in RT_{0,N}$ and $v_{cr}$ and $w_{cr}$ in $V_D^{cr}$, we have
\begin{eqnarray} \label{vww}
a_h(v_{cr},w_{cr}) &=&  =  (A\nabla_h v_{cr}+\btau_{rt}, \nabla_h w_{cr})+(\gradt\btau_{rt}+X_hv_{cr},w_{cr}).
\end{eqnarray}
It follows from \eqref{is}, \eqref{vww}, the Cauchy-Schwarz and discrete Poincar\'e inequalities, and the assumption on $A$, for any $ (\btau_{rt},v_{cr})\in RT_{0,N}\times V^{cr}_D$,
\begin{eqnarray*}
\beta \|\nabla_h v_{cr}\|_0 &\leq & \sup_{w_{cr}\in V^{cr}_D} \dfrac{a_h(v_{cr},w_{cr})}{\|\nabla_h w_{cr}\|_0} = \sup_{w_{cr}\in V^{cr}_D} \dfrac{(A\nabla_h v_{cr}+\btau_{rt}, \nabla_h w_{cr})+(\gradt\btau_{rt}+X_hv_{cr},w_{cr})}{\|\nabla_h w_{cr}\|_0} \\
 &\leq& \sup_{w_{cr}\in V^{cr}_D} \dfrac{\|A^{1/2}\nabla_h v_{cr}+A^{-1/2}\btau_{rt}\|_0\|A^{1/2}\nabla_h w_{cr}\|_0+\|\gradt\btau_{rt}+X_hv_{cr}\|_0\|w_{cr}\|_0}{\|\nabla_h w_{cr}\|_0}\\
& \leq&  C(\|A^{1/2}\nabla_h v_{cr}+A^{-1/2}\btau_{rt}\|_0 + \|\gradt\btau_{rt}+X_hv_{cr}\|_0).
\end{eqnarray*}
The theorem is proved by combining the above result, \eqref{tmp1}, and \eqref{tmp2}.
\end{proof}

\begin{rem}
We present three proofs of the existence and uniqueness of the two-field potential-flux div CR-LSFEM \eqref{lsfem-divl2}. The first proof is based on Assumption \ref{asmp_coef}, ensuring the coerciveness of the original variational problem. The restriction on the mesh size in the first proof is local and explicit. The regularity assumption is not needed for the first proof. However,  the first proof can not be applied to the indefinite problems. On the other hand, the second and the third proofs are based on Assumption \ref{regularity}. It can be applied to more general indefinite cases once the uniqueness and a minimal regularity are assumed. Another restriction of the second and third proofs is that the regularity assumption is global, so the global mesh size $h$ instead of the local mesh size is small enough is needed.
\end{rem}

\section{A Posteriori Estimates of Two-Field Div CR-LSFEM}
\setcounter{equation}{0}

\subsection{A negative result on norm equivalence}
First, we present a negative result on the norm equivalence of the least-squares functional $\cJ^{div}_h$ and the $\tri(\btau,v)\tri$-norm for $v\in W_D^{1+cr}$ and $\btau \in H(\divvr;\O)$. For simplicity, we only discuss the simple case that $A=I$, $\bb = \bzero$, and $c=0$.

\begin{lem}\label{counter_ex}
The following inequality is not true,
\beq
C \tri (\btau, v) \tri ^2  \leq \|\nabla_h v+ \btau\|_0^2 + \|\gradt \btau\|_0^2 \quad \forall (\btau, v)\in H(\divvr;\O)\times W^{1+cr}_D.
\eeq
\end{lem}
\begin{proof}
A counterexample in the two-dimensional case is presented. Let $0\neq w_{cr} \in V^{cr}_D$, but $w_{cr}  \not\in H_D^1(\O)$. For example, choose $w_{cr}(m_F) =1$, where $m_F$ is the mid-point of an interior edge of a finite element mesh and let all other degrees of freedom of $w_{cr}$ be zero.
We have the following Helmholtz decomposition, 
\beq \label{hel1}
 \nabla_h w_{cr} = \nabla \alpha + \gperp \beta \mbox{  and }
\|\nabla_h w_{cr}\|_0^2 = \|\nabla \alpha\|_0^2 + \|\gperp \beta\|_0^2,
\eeq
with $\a\in H^1_D(\O)$ and $\beta\in H^1_N(\O)$. Since $0\neq w_{cr} \in V^{cr}_D$, but $w_{cr}  \not\in H_D^1(\O)$, we get $\gperp \beta \neq 0$.

Let $v = -\alpha +w_{cr} \in W^{1+cr}_D$ and 
$\btau =  \nabla_h w_{cr} - \nabla\alpha = -\nabla_h v = \gperp \beta$. Then $\gradt\btau = \gradt(\gperp \beta) =0$. Thus, we get 
$$
\|\nabla_h v+ \btau\|_0^2 + \|\gradt \btau\|_0^2 =0
\quad \mbox{but}\quad 
\tri (\btau, v) \tri ^2 =
\|\nabla_h v\|_0^2+ \|\btau\|_0^2+  \|\gradt \btau\|_0^2= 2\|\gperp \beta\|_0^2 >0.
$$
We get a contradiction.
\end{proof}

\subsection{A posteriori error estimates}

Lemma \ref{counter_ex} provides a counterexample that
$
\cJ^{div}_{h}(\bsigma_{rt},u_{cr};f)
$ itself cannot be used as a reliable a posteriori error estimator. From the proof of Lemma \ref{counter_ex}, we need to design an error estimator by adding the missing part.

Let $E_h$ be the enriching operator defined in \eqref{enrich}.
Define
\beq
u_c = E_h u_{cr} \in S_{2,D}.
\eeq
We introduce the following six a posteriori error estimators:
\begin{eqnarray*}\label{est1}
\eta_1^2 := \|A^{1/2} \nabla_h(u_{cr} - u_c)\|_0^2 + \cJ^{div}_{h}(\bsigma_{rt},u_{c};f),
&\quad &
\eta_2^2 :=  \|A^{1/2} \nabla_h(u_{cr} - u_c)\|_0^2 + \cJ^{div}_{h}(\bsigma_{rt},u_{cr};f),\\[1mm]
\label{est3}
\eta_3^2 :=\sum_{F\in\cE} \dfrac{1}{h_F} \|\jump{u_{cr}}\|_{0,F}^2
 + \cJ^{div}_{h}(\bsigma_{rt},u_{c};f),
&\quad& 
\eta_4^2 :=\sum_{F\in\cE} \dfrac{1}{h_F} \|\jump{u_{cr}}\|_{0,F}^2
 + \cJ^{div}_{h}(\bsigma_{rt},u_{cr};f),\\[1mm]
 \label{est5}
\eta_5^2 :=\sum_{F\in\cE} {h_F} \|\gamma_{t_F}(u_{cr})\|_{0,F}^2
 + \cJ^{div}_{h}(\bsigma_{rt},u_{c};f),
&\quad&
\eta_6^2 :=\sum_{F\in\cE} {h_F}\|\gamma_{t_F}(u_{cr})\|_{0,F}^2
 + \cJ^{div}_{h}(\bsigma_{rt},u_{cr};f).
\end{eqnarray*}

We have the following reliability and efficiency results.
\begin{thm}
Assume one of the conditions of Theorem \ref{assp} is true.
Let $(\bsigma_{rt},u_{cr})\in RT_{0,N}\times V^{cr}_D$ be the finite element solution to the problem \eqref{cr-lsfem-divl2}.
For $i=1,\cdots,6$, there exist positive constants $C_1$ and $C_2$ such that the following inequalities hold:
\beq
C_1\eta_i \leq \tri (\bsigma-\bsigma_{rt}, u-u_{cr}) \tri
\leq C_2 \eta_i.
\eeq
\end{thm}
\begin{proof}
Due to the norm equivalence \eqref{norm-equ-standard} on $H_N(\divvr;\O)\times H_D^1(\O)$ and the fact that $u_c = E_h u_{cr} \in S_{2,D}\subset H^1_D(\O)$, we have
$$
C_1 \tri (\bsigma-\bsigma_{rt},u-u_c) \tri^2 \leq \cJ^{div}_{h}(\bsigma_h,u_c;f) = \cJ^{div}_{h}(\bsigma-\bsigma_{rt},u-u_c;0)  \leq C_2 \tri (\bsigma-\bsigma_{rt},u-u_c) \tri^2. 
$$
Then by the triangle inequality,
$$
\tri (\bsigma-\bsigma_{rt},u-u_{cr}) \tri^2
\leq  \tri (\bsigma-\bsigma_{rt},u-u_{c}) \tri^2 + C\|A^{1/2} \nabla_h(u_{cr} - u_c)\|_0^2 \leq C\eta_1^2.
$$
The reliability result of $\eta_1$  is proved.

From \eqref{esE2} and \eqref{esE4}, we have
\begin{eqnarray} \label{est111}
C\|\nabla_h(u_{cr}-u_c)\|_0^2 & \leq& \sum_{F\in\cE} \dfrac{1}{h_F} \|\jump{u_{cr}}\|_{0,F}^2\leq C \|\nabla_h(u_{cr}-u_c)\|_0^2,\\ \label{est222}
\sum_{F\in\cE} \dfrac{1}{h_F} \|\jump{u_{cr}}\|_{0,F}^2&\leq& C \|\nabla_h(u-u_{cr})\|_0^2,
\end{eqnarray}
With the  triangle inequality, \eqref{est111}, and \eqref{est222},
\begin{eqnarray*}
\eta_1^2 &=&  \|A^{1/2} \nabla_h(u_{cr} - u_c)\|_0^2 + \cJ^{div}_{h}(\bsigma_{rt},u_{c};f)
\leq  C\|A^{1/2} \nabla_h(u_{cr} - u_c)\|_0^2 +\cJ^{div}_{h}(\bsigma_{rt},u_{cr};f) \\
&\leq& C\|A^{1/2} \nabla_h(u-u_{cr})\|_0^2 +\cJ^{div}_{h}(\bsigma_{rt},u_{cr};f) \leq  C\tri (\bsigma-\bsigma_{rt},u-u_{cr}) \tri^2
\end{eqnarray*}
The efficiency of $\eta_1$ is also proved.

By the triangle inequality, it is easy to see
$$
\eta_1^2 \leq \eta_2^2+\|A^{1/2} \nabla_h(u_{cr} - u_c)\|_0^2 \leq 2\eta_2^2.
$$
Similarly,
$
\eta_2^2 \leq 2\eta_1^2.
$
So $\eta_2$ is both reliable and efficient.

We get from \eqref{est111} that $\eta_3$ and $\eta_4 $ are both reliable and efficient.
The reliability and efficiency of $\eta_5$ and $\eta_6$ can be obtained from \eqref{jj1}.
\end{proof}
\begin{rem}
In this remark, we explain the terms in the error estimators. The error estimator contains three parts. The major part is the residual, that is, $\|\gradt\bsigma_{rt}+X_h u_{cr}-f\|_0$ or $\|\gradt\bsigma_{rt}+X u_{c}-f\|_0$. The other terms represent the fact that the numerical solutions are not in the right spaces. As discussed in \cite{CZ:09,CZ:10a,CYZ:11}, we have
\beq
\bsigma = -A\nabla u \in H(\divvr;\O), \quad u\in H^1(\O), \quad\mbox{and}\quad \nabla u \in H(\curll;\O).
\eeq
The result $\nabla u \in H(\curll;\O)$ is obtained from $u\in H^1(\O)$.
While numerically, 
\beq
-A\nabla_h u_{cr} \not\in H(\divvr;\O), \quad -A\nabla u_{c} \not\in H(\divvr;\O), \quad u_{cr}\not\in H^1(\O), \quad\mbox{and}\quad \nabla_h u_{cr} \not\in H(\curll;\O).
\eeq
The terms $\|A^{-1/2}\bsigma_{rt}+A^{1/2}\nabla u_{cr}\|_0$ and $\|A^{-1/2}\bsigma_{rt}+A^{1/2}\nabla u_{c}\|_0$ measure the distance of the numerical flux from $u_{cr}$ ($A\nabla_h u_{cr}$) or $u_c$ ($A\nabla u_{c}$) to $H(\divvr;\O)$ space.  The term $\sum_{F\in\cE} \dfrac{1}{h_F} \|\jump{u_{cr}}\|_{0,F}^2$ measures the distance of the numerical solution $u_{cr}$ to  $H^1(\O)$ space. The term $\sum_{F\in\cE} {h_F}\|\gamma_{t_F}(u_{cr})\|_{0,F}^2$ measures the distance of the numerical gradient $\nabla_h u_{cr}$ to  $H(\curll;\O)$ space. Due to that fact that $\nabla u \in H(\curll;\O)$ is actually obtained from $u\in H^1(\O)$ and the equivalence  \eqref{jj1}, we can use either  $\sum_{F\in\cE} \dfrac{1}{h_F} \|\jump{u_{cr}}\|_{0,F}^2$ or  $\sum_{F\in\cE} {h_F}\|\gamma_{t_F}(u_{cr})\|_{0,F}^2$ to measure this violation.

From the proof of Lemma \ref{counter_ex}, the curl part of the Helmholtz decomposition \eqref{hel1} is not controlled by the div least-squares functional. This part is exactly the so-called nonconforming error.
The necessity of measuring the violation that $ u_{cr}\not\in H^1(\O)$ or	 $\nabla_h u_{cr} \not\in H(\curll;\O)$  is discussed in the literature of a posteriori error estimates of nonconforming finite elements, see
 \cite{CBJ:02,Ain:05,CZ:10a,CHZ:17cr,CHZ:17,CHZ:20,CHZ:21}.
 
This also hints at how to design a least-squares functional with the norm equivalence on $W_D^{1+cr}\times H(\divvr;\O)$. An equation related to $\nabla u\in H(\curll)$ should be added. We will discuss it in the next section.
\end{rem}

\begin{rem}
Compared to $\eta_i$, $i=3,\cdots,6$, the estimators $\eta_1$ and $\eta_2$ use the matrix $A$, thus they will be more robust with respect to the coefficient.
\end{rem}

\section{Three-Field Potential–Flux–Intensity Div-Curl LSFEM with Nonconforming Approximation for General Elliptic Equations}
\setcounter{equation}{0}

Due to the lack of the norm equivalence in the abstract nonconforming piecewise $H^1$-space $W_D(\cT)$, the potential-flux div least-squares CR method does not have the automatic discrete stability and a built-in a posteriori error estimator with the least-squares functional. Inspired by the a posteriori error analysis discussed in the previous section, we propose three-field formulations, potential-flux-density div-curl least-squares methods. An intensity field and its curl equation are added to the least-squares formulation. In this new formulation, we can prove norm equivalence for the abstract nonconforming piecewise $H^1$-space. Thus, we recover the good qualities of the original LSFEM: automatically discrete stability without a mesh size requirement and regularity assumptions, and a built-in least-squares a posteriori error estimator. 


\subsection{Three-field potential–flux–intensity div-curl least-squares methods}
It is easy to see that only two of the last three equations in \eqref{gen_sys} are independent. Thus, we only need two of them to construct a least-squares functional. We will discuss one case first. The results can be easily generalized to the other two cases; see subsection \ref{twov}.
 
To simplify the notations, let
\begin{eqnarray*}
\Y:=H_N(\divvr;\O)\times H_D(\curllr;\O) \times W_D(\cT) \quad\mbox{and}\quad
 \|(\btau,\bpsi,v)\|_\Y^2 := \|\btau\|_{H(\divvr)}^2+ \|\bpsi\|_{H(\curllr)}^2+\|\nabla_h v\|_0^2.
\end{eqnarray*}
Define the following least-squares functional: For $(\btau,\bpsi,v)\in \Y$,
\beq
\cJ^{divcurl}_{h,1}(\btau,\bpsi,v;f) :=\|A^{-1/2}(\btau + A\nabla_h v)\|^2_0+ \|A^{-1/2}\btau - A^{1/2}\bpsi\|^2_0 +\|\curlt \bpsi\|^2_0+\|\gradt\btau+  X_h v-f\|^2_0 .
\eeq
Then the three-field potential-flux-intensity div-curl least-squares minimization problem in the abstract space is: Find $(\bsigma,\bphi,u)\in \Y$, such that
\beq
\cJ^{divcurl}_{h,1}(\bsigma,\bphi,u;f) = \inf_{(\btau,\bpsi,v)\in \Y}\cJ^{divcurl}_{h,1}(\btau,\bpsi,v;f).
\eeq 
For $(\bchi,\bbbeta,w) \in \Y$ and $(\btau,\bpsi,v) \in  \Y$, define the following  bilinear forms $c_{h}$:
\begin{eqnarray*}
c_{h}((\bchi,\bbbeta,w), (\btau,\bpsi,v))&:=& (A\nabla_h w+ \bchi, \nabla_h v+ A^{-1}\btau)+(\bchi-A\bbbeta, A^{-1}\btau-\bpsi)\\
&&+(\nabla\times \bbbeta, \nabla\times\bpsi)+(\gradt\bchi+X_h w, \gradt\btau+X_h v).
\end{eqnarray*}
Then the least-squares variational problems are: Find $(\bsigma,\bphi,u)\in \Y$,
\beq
c_{h}((\bsigma,\bphi,u), (\btau,\bpsi,v)) = (f, \gradt\btau+X_h v) \quad \forall (\btau,\bpsi,v) \in  \Y.
\eeq
It is clear the exact solution is $(\bsigma,\bphi,u)$.
The following continuity is also easy to verify: 
\beq
c_{h}((\bchi,\bbbeta,w), (\btau,\bpsi,v)) \leq C  \|(\bchi,\bbbeta,w)\|_\Y \|(\btau,\bpsi,v)\|_\Y \quad (\bchi,\bbbeta,w)\in \Y, (\btau,\bpsi,v)\in\Y.
\eeq

\subsection{Coerciveness of three-field div-curl least-squares method}
For a regular mesh $\cT$ with any mesh size, we want to prove the following norm equivalence under the minimal assumption that one of the conditions of Theorem \ref{assp} is true:
\beq
C \|(\btau,\bpsi,v)\|_\Y^2 \leq \cJ^{divcurl}_{h,1}(\btau,\bpsi,v;0) \leq C \|(\btau,\bpsi,v)\|_\Y^2\quad \forall (\btau,\bpsi,v)\in \Y.
\eeq
A careful look into the proofs of the coerciveness of the least-squares methods in \cite{BLP:97,Cai:04,CFZ:15} for the general elliptic equations will find that the compactness argument plays a central role in these proofs. In a compactness argument, the proof by contradiction is used. However, for a method depending on a discrete mesh $\cT$, we cannot use the proof by contradiction to show the coerciveness constant is independent of the mesh. Thus, we must seek a proof without using a mesh-dependent space like $W_D(\cT)$ or $W^{1+cr}_D$.

To this end, we use the Helmholtz decomposition to change the formulation into a mesh-independent setting. Suppose the conditions of Lemmas \ref{hm2D} and \ref{hm3D} hold, by Theorem \ref{hd}, for $v\in W_D(\cT)$, there exists $p \in H^1_D(\O)$ and $\tq\in \Q$ such that 
\beq \label{hmmmm}
A \nabla_h v = A\nabla p + \tcurl \tq \mbox{ and }
\|A^{-1/2}\nabla_h v\|_0^2 = \|A^{1/2}\nabla p\|_0^2 + \|A^{-1/2} \tcurl  \tq\|_0^2.
\eeq
To simplify notations, let
\begin{eqnarray*}
\Z&:=&H_N(\divvr;\O)\times H_D(\curllr;\O) \times H^1_D(\O) \times \Q,\\
\|(\btau,\bpsi,p,\tq)\|_\Z^2 &:=& \|\btau\|_{H(\divvr)}^2+ \|\bpsi\|_{H(\curllr)}^2+\|\nabla p\|_0^2 +\|\tcurl \tq\|_0^2.
\end{eqnarray*}
For $(\btau,\bpsi,p,\tq)\in \Z$, define
\begin{eqnarray*}
L(\btau,\bpsi,p,\tq) &=&\|A^{-1/2}(\btau + A\nabla p +\tcurl \tq)\|^2_0 + \|A^{-1/2}\btau - A^{1/2}\bpsi\|^2_0 \\
	&&+ \|\curlt \bpsi\|^2_0+\|\gradt\btau+  Xp+ \bb\cdot A^{-1}\tcurl\tq\|^2_0.
\end{eqnarray*}
\begin{lem}\label{precoer} Assuming that one of the conditions of Theorem \ref{assp} is true, we have the following coerciveness:
\begin{eqnarray}
C  \|(\btau,\bpsi,p,\tq)\|_\Z^2 \leq L(\btau,\bpsi,p,\tq)\quad\forall (\btau,\bpsi,p,\tq)\in \Z.
\end{eqnarray}
\end{lem}

\begin{proof}
By the standard coerciveness \eqref{norm-equ-standard} and the triangle inequality, we have
\begin{eqnarray}\nonumber
&&C(\|\btau\|_{H(\divvr)}+\|\nabla p\|_0) \leq \|A^{-1/2}\btau + A^{1/2}\nabla p\|_0+\|\gradt\btau+   X p\|_0\\ \label{tpq}
&&\quad\leq \|A^{-\frac12}(\btau + A\nabla p+\tcurl \tq)\|_0+\|\gradt\btau+ Xp+\bb\cdot A^{-1}\tcurl \tq\|_0 + C \|\tcurl \tq\|_0.
\end{eqnarray}
By \eqref{ip1}, \eqref{ip2}, the Poincar\'{e}-Friedrichs inequality \eqref{PF}, Cauchy-Schwarz inequality, and the property of the coefficient matrix $A$, we have
\begin{eqnarray*}\notag
&&\|A^{-1/2}\tcurl \tq\|^2_0
= (A^{-1}\tcurl \tq +\nabla p + A^{-1}\btau, \tcurl \tq)-(A^{-1}\btau-\bpsi, \tcurl \tq)- (\bpsi, \tcurl \tq)\\\notag
&=& (A^{-1}\tcurl \tq +\nabla p + A^{-1}\btau, \tcurl \tq)-(A^{-1}\btau-\bpsi, \tcurl \tq)+ (\nabla\times \bpsi, \tq)\\\notag
&\leq&C(\| A^{-1/2}\tcurl \tq+A^{1/2}\nabla p + A^{-1/2}\btau\|_0 + \|A^{-1/2}\btau - A^{1/2}\bpsi\|_0+\|\nabla\times \bpsi\|_0 )\|A^{-1/2}\tcurl \tq\|_0.
\end{eqnarray*}
Thus, 
\beq\label{qq}
C\|\tcurl \tq\|_0
\leq \| A^{-1/2}\tcurl \tq+A^{1/2}\nabla p + A^{-1/2}\btau\|_0 + \|A^{-1/2}\btau - A^{1/2}\bpsi\|_0+\|\nabla\times \bpsi\|_0.
\eeq
We then get the following inequality:
$$
\|\tcurl \tq\|^2_0 \leq C L(\btau,\bpsi,p,q).
$$
The term $\|\bpsi\|_0$ can be bounded by the triangle inequality:
$$
\|\bpsi\|_0\leq C\|A^{-1/2}\btau - A^{1/2}\bpsi\|_0 + C\|\btau\|_0.
$$
Combined with the above results and using the fact $\|\nabla \times \bpsi\|_0$ is a part of $L(\btau,\bpsi,p,q)$, the lemma is proved.
\end{proof}

\begin{thm} \label{divcurl_coer}
 Assuming that one of the conditions of Theorem \ref{assp} is true, and the conditions of Lemmas \ref{hm2D} and \ref{hm3D} are true, we have the following norm equivalence,
\beq\label{normeqJ1}
C \|(\btau,\bpsi,v)\|_\Y^2 \leq \cJ^{divcurl}_{h,1}(\btau,\bpsi,v;0) \leq C \|(\btau,\bpsi,v)\|_\Y^2  \quad \forall (\btau,\bpsi,v)\in \Y.
\eeq
\end{thm}
\begin{proof}
The upper bound is straightforward with the triangle inequality and the discrete the Poincar\'{e}-Friedrichs inequality \eqref{dPF}.

We will focus on proving the coerciveness.
By the Helmholtz decomposition \eqref{hmmmm}, the result of Lemma \ref{precoer}, the triangle inequality, and the nonconforming Poincar\'{e}-Friedrichs inequality \eqref{dPF}.
\begin{eqnarray*}
&&\|A^{1/2}\nabla_h v\|^2_0 + \|\btau\|^2_{H(\divvr)} + \|\bpsi\|^2_{H(\curllr)}
= \|A^{1/2}\nabla p\|^2_0
+ \|A^{-1/2}\tcurl \tq|^2_0 + \|\btau\|^2_{H(\divvr)} + \|\bpsi\|^2_{H(\curllr)} \\
&\leq& CL(\btau,\bpsi,p,\tq) \\
&=&C\big(\|A^{-1/2}(\btau + A\nabla p + \tcurl \tq)\|^2_0 + \|A^{-1/2}\btau - A^{1/2}\bpsi\|^2_0
+ \|\curlt \bpsi\|^2_0+\|\gradt\btau+  Xp+ \bb\cdot A^{-1}\tcurl\tq\|^2_0\\
&=& C(\|A^{-1/2}(\btau + A\nabla_h v)\|^2_0+ \|A^{-1/2}\btau - A^{1/2}\bpsi\|^2_0 +\|\curlt \bpsi\|^2_0+\|\gradt\btau+  X_h v+c(p-v)\|^2_0)
\\
&\leq& C(\cJ^{divcurl}_{h,1}(\btau,\bpsi,v;0)+ \|p-v\|^2_0 )
\leq C(\cJ^{divcurl}_{h,1}(\btau,\bpsi,v;0)+ \|\nabla p-\nabla_h v\|^2_0)\\
&\leq&C(\cJ^{divcurl}_{h,1}(\btau,\bpsi,v;0)+ \|A^{-1}\tcurl \tq\|_0^2).
\end{eqnarray*}
By \eqref{qq}, we also have
\begin{eqnarray}\notag
C \|A^{-1/2}\tcurl \tq\|_0^2
&\leq&\|A^{-1/2}\tcurl \tq+A^{1/2}\nabla p + A^{-1/2}\btau\|_0^2 + \|A^{-1/2}\btau - A^{1/2}\bpsi\|_0^2+\|\nabla\times \bpsi\|_0^2  \\ \nonumber 
&\leq& \cJ^{divcurl}_{h,1}(\btau,\bpsi,v;0).
\end{eqnarray}
The theorem is then proved.
\end{proof}
%

%

\subsection{Three-field potential–flux–intensity div-curl nonconforming least-squares finite element methods}
Let 
$$
\Y_h:=RT_{0,N}\times N_{0,D} \times V^{cr}_D.
$$ We have $\Y_h\subset \Y$.
The three-field div-curl nonconforming LSFEM is to find $(\bsigma_h,\bphi_h,u_h)\in \Y_h$, such that
\beq \label{lsfem_m}
\cJ^{divcurl}_{h,1}(\bsigma_h,\bphi_h,u_h;f) = \inf_{(\btau_h,\bpsi_h,v_h)\in \Y_h}\cJ^{divcurl}_{h,1}(\btau_h,\bpsi_h,v_h;f), 
\eeq 
or equivalently in a weak form:
Find $(\bsigma_h,\bphi_h,u_h)\in \Y_h$, such that
\beq \label{lsfem_v}
c_{h}((\bsigma_h,\bphi_h,u_h), (\btau_h,\bpsi_h,v_h)) = (f, \gradt\btau_h+X_h v_h) \quad \forall (\btau_h,\bpsi_h,v_h) \in  \Y_h.
\eeq
It is easy to see we have the error equation,
$$
c_{h}((\bsigma-\bsigma_h,\bphi-\bphi_h,u-u_h), (\btau_h,\bpsi_h,v_h))=0 \quad \forall (\btau_h,\bpsi_h,v_h) \in  \Y_h,
$$
and the a priori error estimate:
Assuming $u\in H^2(\O)$ (thus $\bphi\in H^1(\O)$ and we always have $\nabla\times \bphi=0$), $\bsigma|_K \in H^{1}(K)$, and $\gradt\bsigma|_K \in H^{1}(K)$,  for $K\in \cT$, then
\begin{eqnarray*}
\|(\bsigma-\bsigma_h,\bphi-\bphi_h,u-u_h)\|_\Y &\leq& C\inf_{(\btau_h,\bpsi_h,v_h) \in  \Y_h}\|(\bsigma-\btau_h,\bphi-\bpsi_h,u-v_h)\|_\Y 
\\
&\leq &C h (\|u\|_{2}+ \|\bsigma\|_{1,h} +  \|\gradt\bsigma\|_{1,h}),
\end{eqnarray*}
where $\|v\|_{1,h}^2 = \sum_{K\in\cT} \|v\|_{1,K}^2$.
\begin{rem}
Using the more refined analysis of local interpolations or local quasi-interpolations of N\'ed\'elec elements \cite{FE1,EG:17}, we can also show a local optimal error estimate.
\end{rem}

Let $(\bsigma_h,\bphi_h,u_h)\in\Y_h$ be the numerical solution of \eqref{lsfem_m}. Define the following a posteriori error estimator:
\beq
\zeta ^2: = \cJ^{divcurl}_{h,1}(\bsigma_h,\bphi_h,u_h;f). 
\eeq
Due to the norm equivalence \eqref{normeqJ1}, we have the global reliability and efficiency of $\zeta$:
\beq
C_1 \zeta ^2\leq  \|(\bsigma-\bsigma_h,\bphi-\bphi_h,u-u_h )\|_\Y^2 \leq
	C_2 \zeta ^2.
\eeq

\begin{rem}
Since $-\bphi$ and $-A^{-1}\bsigma$ are both $\nabla u$, the term $X_h v$ in the above definitions of the least-squares functional can also be replaced by
\begin{eqnarray}
G_{h,rst}(\bpsi,\btau,v) & := & \bb\cdot(r\nabla_h v - s\bpsi -t A^{-1}\btau) + cv \quad \forall (\btau,\bpsi,v)\in \Y,
\end{eqnarray}
where $(r,s,t) \in [0,1]^3$ and $r+s+t =1$. We can get similar results of norm equivalence and a priori and a posteriori error estimates.
\end{rem}


\subsection{Two variants of three-field potential–flux–intensity div-curl nonconforming least-squares method}\label{twov}
We can also use other combinations in the first-order system \eqref{gen_sys} to define least-squares functionals.
Define the following least-squares functionals: For $(\btau,\bpsi,v)\in \Y$,
\begin{eqnarray*}
\cJ_{h,2}(\btau,\bpsi,v;f) &:=&\|A^{1/2}(\nabla_h v + \bpsi)\|^2_0+ \|A^{-1/2}\btau - A^{1/2}\bpsi\|^2_0+ \|\curlt \bpsi\|^2_0+\|\gradt\btau+  X_h v-f\|^2_0,\\ 
\cJ_{h,3}(\btau,\bpsi,v;f) &:=&\|A^{-1/2}(\btau + A\nabla_h v)\|^2_0+ \|A^{1/2}(\nabla_h v + \bpsi)\|^2_0+ \|\curlt \bpsi\|^2_0+\|\gradt\btau+  X_h v-f\|^2_0.
\end{eqnarray*}
We then prove the equivalence of different $\cJ_{h,i}$. We use the notation $B \approx D$ to denote there exist constants $C_1$ and $C_2$, such that $C_1 D\leq B\leq C_2 D$.    
\begin{lem}
The following equivalences are true:
\begin{eqnarray} \label{equva}
\cJ^{divcurl}_{h,1}(\btau,\bpsi,v;0) \approx  \cJ_{h,2}(\btau,\bpsi,v;0)  \approx \cJ_{h,3}(\btau,\bpsi,v;0)  \quad \forall (\btau,\bpsi,v)\in \Y.
\end{eqnarray}
\end{lem}
\begin{proof}
By the triangle inequality, we have
$$
\|A^{1/2}(\nabla_h v+ \bpsi)\|_0\leq \|A^{-1/2}\btau + A^{1/2}\nabla_h v\|_0+\|A^{-1/2}\btau - A^{1/2}\bpsi\|_0.
$$
Thus $\cJ_{h,2}(\btau,\bpsi,v;0) \leq C\cJ^{divcurl}_{h,1}(\btau,\bpsi,v;0) $ is true for all $(\btau,\bpsi,v)\in \Y$. The other results can be proved by similar arguments.
\end{proof}
We thus can define potential–flux–intensity div-curl nonconforming least-squares methods based on functionals $\cJ^{divcurl}_{h,2}$ and $\cJ^{divcurl}_{h,3}$. Due to the equivalence \eqref{equva}, all results of properties related to $\cJ^{divcurl}_{h,1}$ can be generalized to methods defined by $\cJ^{divcurl}_{h,2}$ or $\cJ^{divcurl}_{h,3}$.

\section{A Two-Field Potential–Flux Div-Curl LSFEM with Nonconforming Approximation}
\setcounter{equation}{0}

In this section, we discuss the application and restrictions of the original potential-flux div-curl least-squares method \cite{CMM:97} in a nonconforming setting. When the domain is nice, and the coefficient is sufficiently smooth, the original formulation introduced in \cite{CMM:97} can be used in the nonconforming case. But this two-field formulation can cause serious problems when the conditions on the domain and coefficients are not satisfied. 
\subsection{Two-field potential–flux div-curl first-order system}
In the original paper \cite{CMM:97}, the intensity $\bphi=-\nabla u$ is not introduced as an independent variable; instead, the following first-order system is discussed.
\begin{equation} \label{gen_sys2}
\gradt \bsigma +  Xu   = f  \mbox{ in } \O,  \quad
 A\nabla u+ \bsigma   = 0  \mbox{ in } \O, \quad
\curlt (A^{-1}\bsigma) = 0  \mbox{ in } \O,
\end{equation}
with boundary conditions
$u=0 \mbox{ on }\Gamma_D$, $\gamma_{t}(A^{-1}\bsigma)=0 \mbox{ on }\Gamma_D$, and 
$\bn\cdot\bsigma=0 \mbox{ on } \Gamma_N.
$
Let 
$$
\bSigma := \{\btau: \btau \in H_N(\divvr;\O), \; A^{-1}\btau\in H_D(\curll;\O)\}
\quad\mbox{with  }
\|\btau\|_{\bSigma}^2 := \|\btau\|_0^2 + \|\gradt\btau\|_0^2 +\|\nabla\times(A^{-1}\btau)\|_0^2.
$$ 
Then the exact flux $\bsigma\in \bSigma$.

In this section, we assume the following assumption is true. Detailed discussion on this norm equivalence can be found in \cite{CMM:97}, see also discussion in Section A.3.1 of \cite{BG:09}.
\begin{assumption} \label{cont}
We assume that $A\in C^{1,1}$ and the domain $\O$ is nice enough to guarantee 
that $\bSigma$ is algebraically and topologically included in $H^1(\O)^d$, that is,  the following norm equivalence holds:
\beq
C_1\|\btau\|_1 \leq \|\btau\|_{\bSigma} \leq C_2\|\btau\|_1 \quad \forall \btau\in \bSigma.
\eeq
\end{assumption}
Here, the domain is nice enough means one of the following is true (Theorem A.8 of \cite{BG:09}):
\begin{itemize}
\item $\p\O$ is of class $C^{1,1}$,
\item $\p\O$ is piecewise smooth with no reentrant corners for $d=2$,
\item $\O$ is a convex polyhedron for $d=3$.
\end{itemize}

\subsection{A two-field potential–flux div-curl least-squares method}
Now, the  two-field potential-flux div-curl least-squares functional is given by
\beq
\cG_h(\btau,v;f) := \|A^{-1/2}\btau+A^{1/2}\nabla_h v\|_0^2+\|\gradt\btau+X_hv-f\|_0^2+\|\nabla\times (A^{-1}\btau)\|_0^2 \quad \forall (\btau,v)\in \bSigma\times W_D(\cT).
\eeq
Then the potential-flux div-curl least-squares minimization problem in the abstract space is: 
\beq\label{lspf1}
\mbox{Find  } (\bsigma,u)\in \bSigma\times W_D(\cT), \mbox{  such that  }
\cG_h(\bsigma,u;f) = \inf_{(\btau,v)\in \bSigma\times W_D(\cT)}\cG_h(\btau,v;f).
\eeq 
For $(\bchi,w) \in  \bSigma\times W_D(\cT) $ and $(\btau,v) \in   \bSigma\times W_D(\cT)$, define the following  bilinear forms $d_{h}$:
$$
d_{h}((\bchi,w), (\btau,v)) := (A\nabla_h w+ \bchi, \nabla_h v+ A^{-1}\btau)+(\nabla\times (A^{-1}\bchi), \nabla\times(A^{-1}\btau))+(\gradt\bchi+X_h w, \gradt\btau+X_h v).
$$
Then the  two-field potential-flux div-curl least-squares variational problems is: Find $(\bsigma,u)\in \bSigma\times W_D(\cT)$,
\beq \label{lspf2}
d_{h}((\bsigma,u), (\btau,v)) = (f, \gradt\btau+X_h v) \quad \forall (\btau,v) \in  \bSigma\times W_D(\cT).
\eeq
We use almost the same argument of the three-field potential-flux-curl least-squares method to prove the coerciveness of $\cG_h(\btau,v;0)$.  We assume the same Helmholtz decomposition \eqref{hmmmm} holds.
Let
\begin{eqnarray*}
\T := \bSigma \times H^1_D(\O) \times \Q \quad\mbox{and}\quad
\|(\btau,p,\tq)\|_\T^2 := \|\btau\|_{\bSigma}^2+\|\nabla p\|_0^2 +\|\tcurl \tq\|_0^2.
\end{eqnarray*}
For $(\btau,p,\tq)\in \T$, define
\begin{eqnarray*}
M(\btau,p,\tq) :=\|A^{-1/2}(\btau + A\nabla p +\tcurl \tq)\|^2_0
	+ \|\curlt (A^{-1}\btau)\|^2_0+\|\gradt\btau+  Xp+ \bb\cdot A^{-1}\tcurl\tq\|^2_0.
\end{eqnarray*}
\begin{lem}   Assuming that one of the conditions of Theorem \ref{assp} is true and Assumption \ref{cont} is true, we have the following coerciveness:
\begin{eqnarray}
C  \|(\btau,p,\tq)\|_\T^2 \leq M(\btau,p,\tq)\quad\forall (\btau,p,\tq)\in \T.
\end{eqnarray}
\end{lem}

\begin{proof}
By \eqref{ip1}, \eqref{ip2}, the Poincar\'{e}-Friedrichs inequality \eqref{PF}, the Cauchy-Schwarz inequality, and the property of the coefficient matrix $A$, we have
\begin{eqnarray*}\notag
\|A^{-1/2}\tcurl \tq\|^2_0
&=& (A^{-1}\tcurl \tq +\nabla p + A^{-1}\btau, \tcurl \tq)-(A^{-1}\btau, \tcurl \tq)\\\notag
&=& (A^{-1}\tcurl \tq +\nabla p + A^{-1}\btau, \tcurl \tq)+ (\nabla\times ((A^{-1}\btau), \tq)\\\notag
&\leq&C(\| A^{-1/2}\tcurl \tq+A^{1/2}\nabla p + A^{-1/2}\btau\|_0 +\|\nabla\times (A^{-1}\btau)\|_0 )\|A^{-1/2}\tcurl \tq\|_0.
\end{eqnarray*}
Then, we get
$C\|\tcurl \tq\|_0
\leq \| A^{-1/2}\tcurl \tq+A^{1/2}\nabla p + A^{-1/2}\btau\|_0 +\|\nabla\times (A^{-1}\btau)\|_0$.
Thus,  $\|\tcurl \tq\|^2_0 \leq C M(\btau,p,\tq)$. By \eqref{tpq}, the lemma is then proved.
\end{proof}
The following theorem can be easily proved by using the same argument as Theorem \ref{divcurl_coer}.
\begin{thm}
 Assuming that one of the conditions of Theorem \ref{assp} is true,  Assumption \ref{cont} is true, and the conditions of Lemmas \ref{hm2D} and \ref{hm3D} are true, we have the following norm equivalence,
\beq\label{neq_sd}
C_1( \|\btau\|_\bSigma ^2  + \|\nabla_h v\|_0^2)  \leq \cG_{h}(\btau,v;0) \leq C_2( \|\btau\|_\bSigma^2  + \|\nabla_h v\|_0^2)  \quad \forall (\btau,v)\in \bSigma \times W_D(\cT).
\eeq
\end{thm}

With  Assumption \ref{cont}, we can use standard $H^1$-conforming finite element to approximate $\bsigma$. For simplicity, we use the linear finite element space. Define
$$
\bSigma_h =\{ \btau \in C^0(\O)^d: \btau|_K \in P_1(K)^d \;\forall K\in\cT, \btau \in \bSigma\}.
$$
Then the potential–flux div-curl nonconforming LSFEM of \eqref{lspf1} or  \eqref{lspf2} is:
Find $(\bsigma_h,u_h)\in \bSigma_h\times V_D^{cr}$,
\beq \label{lspf2fem}
d_{h}((\bsigma_h,u_h), (\btau_h,v_h)) = (f, \gradt\btau+X_h v) \quad \forall (\btau_h,v_h) \in \bSigma_h\times V_D^{cr}.
\eeq
We then have the following a priori error estimate as the classic div-curl LSFEM problem (Theorem 3.1 of \cite{CMM:97}): 
\begin{thm}
Assume $u\in H^{1+\alpha}(\O)$ and $\bsigma\in H^{1+\alpha}(\O)^d$ and let $(\bsigma_h,u_h)\in \bSigma_h\times V_D^{cr}$ be the numerical solution of \eqref{lspf2fem}. Then
\begin{eqnarray*}
\|\bsigma-\bsigma_h\|_1  + \|\nabla_h (u-u_h)\|_0
&\leq &C h^{\alpha}(\|u\|_{1+\alpha}+ \|\bsigma\|_{1+\alpha}).
\end{eqnarray*}
\end{thm}
Define the following a posteriori error estimator:
\beq \label{xi}
\xi  ^2: = \cG_{h}(\bsigma_h,u_h;f),
\eeq
where $(\bsigma_h,u_h)\in \bSigma_h\times V_D^{cr}$ is the numerical solution of \eqref{lspf2fem}. Due to the norm equivalence \eqref{neq_sd}, we immediately have the global reliability and efficiency of $\xi$:
\beq
C_1 \xi\leq  \|\bsigma-\bsigma_h\|_1  + \|\nabla_h (u-u_h)\|_0 \leq 
	C_2 \xi.
\eeq
\subsection{Restrictions of the two-field potential–flux div-curl least-squares formulation}
We discuss the reactions of the two-field potential–flux div-curl least-squares formulation due to the violation of Assumption \ref{cont}.

The first restriction is about the domain even when $A$ is $C^{1,1}$, for example, $A=I$. We have the following famous negative result \cite{Cos:91} that when $d=3$, $\{\btau\in H(\divvr;\O):\btau\cdot\bn=0 \mbox{ on }\p\O \}\cap H^1(\O)^3$ and $\{\btau\in H(\curll;\O):\btau\times\bn=0 \mbox{ on }\p\O \}\cap H^1(\O)^3$ are closed, infinite-codimensional subspaces of $\{\btau\in H(\divvr;\O):\btau\cdot\bn=0 \mbox{ on }\p\O \}\cap H(\curll;\O)$ and $\{\btau\in H(\curll;\O):\btau\times\bn=0 \mbox{ on }\p\O \}\cap H(\divvr;\O)$, respectively. For such cases, $H^1$-conforming finite elements can not be used to approximate the flux, see detailed discussion in Section B.2.2 of \cite{BG:09}.

The second restriction is about the possible discontinuity of coefficient $A$. As discussed in \cite{CZ:09,CZ:10a,CYZ:11,CHZ:17zz}, for the exact solution, we have
\beq
\bsigma = -A\nabla u \in H(\divvr;\O) \quad\mbox{and}\quad \nabla u \in H(\curll;\O).
\eeq
But, for a discontinuous $A$, we also have
\beq
\bsigma = -A\nabla u\not\in H^1(\O)^d.
\eeq
Thus for a discontinuous $A$, $\bSigma_h$ or any $H^1$-conforming finite element space is a wrong approximation space for $\bsigma$. The error estimator \eqref{xi} will never be zero even the numerical solution $u_h=u$ is exact; see examples given at \cite{CHZ:17zz}. 

In conclusion, the two-field potential–flux div-curl least-squares formulation should be used very carefully for both the conforming and nonconforming cases. 

\section{Concluding Remarks}
We present two least-squares formulations for the general second-order elliptic equations with nonconforming finite element approximation. One is the two-field potential-flux div formulation, and the other is the three-field potential-flux-intensity div-curl formulation. For the two-field div formulation, the CR-LSFEM does not have the norm equivalence in the abstract setting and thus does not have automatic discrete coerciveness and built-in a posteriori error estimates. The three-field formulation recovers the norm equivalence in the abstract setting and the good properties of the least-squares method. Furthermore, examining the proof carefully, it is easy to find that the three-field potential-flux-intensity div-curl formulation has the potential to be generalized to other non-traditional nonconforming least-squares approximation, for example, the multiscale finite elements \cite{HWC:99,HZZ:14} and the immersed finite elements \cite{Li:98}. In this paper, we do not discuss the formulations with jump-stabilizations. In a future paper, we plan to discuss the least-squares methods with jump-stabilizations for nonconforming and discontinuous finite element approximations for the possible indefinite second-order linear elliptic equations.

\bibliographystyle{plain}
\bibliography{../bib/szhang}
\end{document}

%% file: defs.tex
\newcommand{\BX}{{\bf X}}
\newcommand{\cv}{{\cal V}}
\newcommand{\cW}{{\cal W}}
\newcommand{\co}{{\cal O}}

\renewcommand{\theequation}{\thesection.\arabic{equation}}
\def\@eqnnum{{\reset@font\rm (\theequation)}}

\def\abstract{
\advance \rightskip by 10mm
\advance \leftskip by 10mm
\vspace{-0.8em}
\noindent
\small{\bf Abstract.}
}
\def\endabstract{\par\normalsize\rm}

\def\Xint#1{\mathchoice
{\XXint\displaystyle\textstyle{#1}}%
{\XXint\textstyle\scriptstyle{#1}}%
{\XXint\scriptstyle\scriptscriptstyle{#1}}%
{\XXint\scriptscriptstyle\scriptscriptstyle{#1}}%
\!\int}
\def\XXint#1#2#3{{\setbox0=\hbox{$#1{#2#3}{\int}$}
\vcenter{\hbox{$#2#3$}}\kern-.5\wd0}}
\def\ddashint{\Xint=}
\def\dashint{\Xint-}

\def\a{\alpha}
\def\b{\beta}
\def\d{\delta}\def\D{\Delta}
\def\e{\epsilon}
\def\g{\gamma}\def\G{\Gamma}
\def\k{\kappa}
\def\lam{\lambda}\def\Lam{\Lambda}
\renewcommand\o{\omega}\renewcommand\O{\Omega}
\def\s{\sigma}\def\S{\Sigma}
\renewcommand\t{\theta}\def\vt{\vartheta}
\newcommand{\vphi}{\varphi}
\def\z{\zeta}

\newcommand{\tsigma}{\tilde{\s}}
\newcommand{\tbsigma}{\tilde{\bsigma}}
\def\te{\tilde{\e}}
\def\tu{\tilde{u}}

\newcommand{\bchi}{\mbox{\boldmath$\chi$}}
\newcommand{\bdelta}{\mbox{\boldmath$\delta$}}
\newcommand{\bepsilon}{\mbox{\boldmath$\epsilon$}}
\newcommand{\bfeta}{\mbox{\boldmath$\eta$}}
\newcommand{\bgamma}{\mbox{\boldmath$\gamma$}}
\newcommand{\bomega}{\mbox{\boldmath$\omega$}}
\newcommand{\bvphi}{\mbox{\boldmath$\varphi$}}
\newcommand{\bphi}{\mbox{\boldmath$\phi$}}
\newcommand{\bPhi}{\mbox{\boldmath$\Phi$}}
\newcommand{\bpsi}{\mbox{\boldmath$\psi$}}
\newcommand{\bPsi}{\mbox{\boldmath$\Psi$}}
\newcommand{\bsigma}{\mbox{\boldmath$\sigma$}}
\newcommand{\btau}{\mbox{\boldmath$\tau$}}
\newcommand{\bxi}{\mbox{\boldmath$\xi$}}
\newcommand{\brho}{\mbox{\boldmath$\rho$}}
\newcommand{\bbeta}{\mbox{\boldmath$\beta$}}
\newcommand{\bzeta}{\mbox{\boldmath$\zeta$}}

\def\bk{\boldsymbol{\kappa}}
\def\bmu{\boldsymbol\mu}
\def\bxi{\boldsymbol{\xi}}
\def\bz{\boldsymbol{\zeta}}

\def\ba{{\bf a}}
\def\bb{{\bf b}}
\def\bc{{\bf c}}
\def\be{{\bf e}}
\def\bff{{\bf f}}
\def\bg{{\bf g}}
\def\bn{{\bf n}}
\def\bp{{\bf p}}
\def\bq{{\bf q}}
\def\bs{{\bf s}}
\def\bt{{\bf t}}
\def\bu{{\bf u}}
\def\bv{{\bf v}}
\def\bw{{\bf w}}
\def\bx{{\bf x}}
\def\by{{\bf y}}
\def\bzz{{\bf z}}

\def\bD{{\bf D}}
\def\bE{{\bf E}}
\def\bF{{\bf F}}
\def\bH{{\bf H}}
\def\bJ{{\bf J}}
\def\bV{{\bf V}}
\def\bU{{\bf U}}
\def\bW{{\bf W}}
\def\bX{{\bf X}}
\def\bY{{\bf Y}}

\def\cA{{\cal A}}
\def\cC{{\cal C}}
\def\cD{{\cal D}}
\def\cE{{\cal E}}
\def\cF{{\cal F}}
\def\cG{{\cal G}}
\def\cI{{\cal I}}
\def\cJ{{\cal J}}
\def\cK{{\cal K}}
\def\cL{{\cal L}}
\def\cO{{\cal O}}
\def\cP{{\cal P}}
\def\cQ{{\cal Q}}
\def\cR{{\cal R}}
\def\cS{{\cal \Sigma}}
\def\cT{{\cal T}}
\def\cU{{\cal U}}
\def\cV{{\cal V}}

\def\scT{{_\cT}}
\def\sD{{_D}}
\def\sE{{_E}}
\def\sF{{_F}}
\def\sFz{{_{F_z}}}
\def\sK{{_K}}
\def\sI{{_I}}
\def\sb{{_b}}
\def\sN{{_N}}

\def\curl{{{\bf curl} \ }}
\def\rot{{\mbox{rot}\ }}
\def\BPI{{\bf \Pi}}

\def\cth{\cT_h}
\def\ctH{\cT_H}

\def\tJ{\tilde{\J}}

\def\hK{\widehat{K}}
\def\hx{\widehat{x}}
\def\hy{\widehat{y}}
\def\bhv{\widehat{\bv}}

\def\l{\ell}
\def\bl{\boldsymbol{\ell}}
\def\col{\colon}
\def\f12{\frac12}
\def\dfrac{\displaystyle\frac}
\def\dint{\displaystyle\int}
\def\nab{\nabla}
\def\p{\partial}
\def\sm{\setminus}
\def\dsum{\displaystyle\sum}
\newcommand{\pp}[2]{\frac{\partial {#1}}{\partial {#2}}}
\def\bzero{{\bf 0}}

\def\divv{\nab\cdot}
\def\divx{\nab_x\cdot}
\def\divtx{\nab_{t,x}\cdot}
\def\nabx{\nab_x}

\newcommand{\grad}{\nabla}
\newcommand{\curlt}{{\nabla \times}}
\newcommand{\gperp}{\nabla^{\perp}}
\newcommand{\gradt}{\nabla\cdot}

\def\forallqq{\quad\forall\,}
\def\aph{A^{1/2}}
\def\amh{A^{-1/2}}

\def\osc{{\rm osc \, }}

\def\Im{{\rm Im}}
\newcommand{\tr}{{\rm tr}}
\def\divvr{{\rm div}}
\def\curllr{{\rm curl}}
\def\curll{{\rm curl}}
\def\curl{{\bf curl}}
\newcommand{\bgrad}{{\bf grad}}
\newcommand\diam{\mathrm{diam\,}}
\renewcommand\Im{\mathrm{Im\,}}
\def\Span{\mbox{Span}}
\def\supp{\mbox{supp\,}}
\newcommand{\trace}{{\rm trace}}

\newcommand{\tri}{|\!|\!|}
\newcommand{\ljump}{\lbrack\!\lbrack}
\newcommand{\rjump}{\rbrack\!\rbrack}
\newcommand{\bdm}{\begin{displaymath}}
\newcommand{\edm}{\end{displaymath}}
\newcommand{\beq}{\begin{equation}}
\newcommand{\eeq}{\end{equation}}
\newcommand{\beqa}{\begin{eqnarray}}
\newcommand{\eeqa}{\end{eqnarray}}
\newcommand{\beqas}{\begin{eqnarray*}}
\newcommand{\eeqas}{\end{eqnarray*}}
\newcommand{\ul}{\underline}
\newcommand{\wh}{\widehat}
\newcommand{\la}{\langle}
\newcommand{\ra}{\rangle}

\newcommand{\Lt}{L^2(\Omega)}
\newcommand{\Lts}{L^2(\Omega)^2}
\newcommand{\Ltc}{L^2(\Omega)^3}
\newcommand{\Ho}{H^1(\Omega)}
\newcommand{\Hoh}{H^1(\wh{\Omega})}
\newcommand{\Hoi}{H^1(\Omega_i)}
\newcommand{\Hos}{H^1(\Omega)^2}
\newcommand{\Hoc}{H^1(\Omega)^3}
\newcommand{\Hoch}{H^1(\wh{\Omega})^3}
\newcommand{\Hoci}{H^1(\Omega_i)^3}
\newcommand{\Hoz}{H^1_0(\Omega)}
\newcommand{\Ht}{H^2(\Omega)}
\newcommand{\Hti}{H^2(\Omega_i)}
\newcommand{\Hts}{H^2(\Omega)^2}
\newcommand{\Htc}{H^2(\Omega)^3}
\newcommand{\Htz}{H^0(\Omega)}
\newcommand{\Hh}{H^{1/2}(\Gamma)}
\newcommand{\Hhi}{H^{1/2}(\Gamma_i)}
\newcommand{\Hmh}{H^{-1/2}(\Gamma)}
\newcommand{\Hdiv}{H(\divvr;\,\Omega)}
\newcommand{\Hdivh}{H(\divv;\,\wh \Omega)}
\newcommand{\hcurl}{H(\curl\,A;\,\Omega)}
\newcommand{\Hcurl}{H(\curll\,A;\,\Omega)}
\newcommand{\Hcrl}{H(\curll\,;\,\Omega)}
\newcommand{\hcrl}{H(\curl\,;\,\Omega)}
\newcommand{\Hcrlh}{H(\curll\,;\,\wh\Omega)}
\newcommand{\hcrlh}{H(\curl\,;\,\wh\Omega)}
\newcommand{\Wdiv}{\BW_0(\mbox{\divv}\,;\,\Omega)}
\newcommand{\Wcurl}{\BW_0(\mbox{\curl}\,A;\,\Omega)}
\newcommand{\WcrossV}{\BW \times V}